\documentclass[11pt]{article}
\usepackage{latexsym}
\usepackage{amssymb}
\usepackage{amsmath}
\usepackage{amsthm}
\usepackage{bm}
\theoremstyle{plain}
\newtheorem{thm}{Theorem}[section]
\newtheorem{lem}[thm]{Lemma}

\newtheorem{rem}[thm]{Remark}
\newtheorem{defin}[thm]{Definition}

\newtheorem{pro}[thm]{Proposition}
\newtheorem{assumption}[thm]{Assumption}

\newcommand{\RR}{{\mathbb R}}

\newcommand{\ep}{\varepsilon}

\newcommand{\tr}{{\rm tr}}

\numberwithin{equation}{section}
\begin{document}
\newcounter{aaa}
\newcounter{bbb}
\newcounter{ccc}
\newcounter{ddd}
\newcounter{eee}
\newcounter{xxx}
\newcounter{xvi}
\newcounter{x}
\setcounter{aaa}{1}
\setcounter{bbb}{2}
\setcounter{ccc}{3}
\setcounter{ddd}{4}
\setcounter{eee}{32}
\setcounter{xxx}{10}
\setcounter{xvi}{16}
\setcounter{x}{38}
\title
{Reflected rough
differential equations\footnote{to appear in Stochastic Processes and
their applications, doi: 10.1016/j.spa.2015.03.008}}
\author{Shigeki Aida
\\
Mathematical Institute\\
Tohoku University,
Sendai, 980-8578, JAPAN\\
e-mail: aida@math.tohoku.ac.jp}
\date{}
\maketitle
\begin{abstract}
In this paper, we study reflected differential equations 
driven by continuous paths with finite $p$-variation ($1\le p<2$) and
$p$-rough paths ($2\le p<3$) on domains in Euclidean spaces
whose boundaries may not be smooth.
We define reflected rough differential
equations and prove the existence of a solution.
Also we discuss the relation between the solution 
to reflected stochastic differential equation 
and reflected rough differential equation
when the driving process is a Brownian motion.

{\it Keywords:}\, reflected stochastic differential equation,
rough path, Skorohod equation
\end{abstract}

\section{Introduction}

In \cite{aida-sasaki}, we proved the
strong convergence of
the Wong-Zakai approximations of the solutions 
to reflected stochastic differential equations
defined on domains in Euclidean spaces
whose boundaries may not be smooth.
The driving stochastic process in the equation is a
Brownian motion.
Recently, many researchers have been studying 
differential equations driven by more general stochastic 
processes and irregular paths.
This is due to the development 
of rough path theory which gives new meaning of stochastic integrals.
The aim of this paper is to study 
reflected differential equations driven by rough paths
and prove the existence of solutions.
We use the Euler approximation of 
the differential equations by modifying
the idea of Davie~\cite{davie}.
When the equation has reflection term, the Euler approximation 
becomes an implicit Skorohod equation and it is not trivial 
to see the existence of the solutions.
Hence, we need stronger assumptions than those given in
\cite{aida-sasaki}
on the boundary of the domain
to prove the existence of solutions.
At the moment, we neither have uniqueness of solutions nor continuity theorem
with respect to driving paths.

The paper is organized as follows.
In Section 2, we introduce several conditions on the boundary 
under which reflected rough differential equations are studied
and prepare necessary lemmas.
In Section 3, we study the reflected differential equations
driven by continuous path of finite $p$-variation with $1\le p<2$.
The meaning of the integral in this equation 
is justified by the Young integrals.
We prove the existence of solutions by using Davie's
approach~\cite{davie}.
This problem was already studied when $D$ is a half space in
\cite{ferrante-rovira}.
Our existence theorem is valid for more general domains.
In Section 4, we study the case where the driving path is
$p$-rough path with $2\le p<3$.
In this case, we consider stronger assumptions than that in previous sections.
First, we define reflected rough differential equations
and prove the existence of a solution and give some estimates of the solution.
Also we explain the reason of the difficulty to prove the uniqueness of
solutions and continuity theorems with respect to driving rough paths.
At the end of this section, we prove the existence of a measurable
solution mapping for geometric rough paths.
In Section 5, we go back to reflected SDEs driven by Brownian
motion.
We explain the relation between the solution to
reflected rough differential equation
and the solution to reflected stochastic differential equation.

\section{Preliminary}

First, we prepare necessary definitions and results 
for our purposes.
The following conditions on the connected domain $D\subset \RR^d$
are standard assumptions for reflected SDE and can be found 
in \cite{lions-sznitman, saisho, tanaka} and we will study our equations
on domains which satisfy these conditions.
We will introduce other conditions later.
For other references of reflected SDEs related with this paper,
we refer the readers to
\cite{aida-sasaki, zhang, doss, dupuis-ishii, evans-stroock,
pettersson, ren-xu1, ren-xu2, slominski1, slominski2}.
In \cite{aida},
we study Wong-Zakai approximations (\cite{wong-zakai})
in the two cases,
(i) the domain is convex,
(ii) the conditions
(A) and (B) are satisfied
which are not contained in the result in 
\cite{aida-sasaki}.

Recall that the set ${\cal N}_x$ 
of inward unit normal vectors at the boundary
point $x\in \partial D$ is defined by
\begin{align*}
 {\cal N}_x&=\cup_{r>0}{\cal N}_{x,r},\\
{\cal N}_{x,r}&=\left\{{\bm n}\in \RR^d~|~|{\bm n}|=1,
 B(x-r{\bm n},r)\cap D=\emptyset\right\},
\end{align*}
where $B(z,r)=\{y\in \RR^d~|~|y-z|<r\}$, $z\in \RR^d$, $r>0$.

\begin{defin}\label{basic definition}
\begin{enumerate}
 \item[{\rm (A)}] 
 There exists a constant
$r_0>0$ such that
\begin{align}
{\cal N}_x={\cal N}_{x,r_0}\ne \emptyset \quad \mbox{for any}~x\in
 \partial D.\nonumber
\end{align}

\item[{\rm (B)}]
There exist constants $\delta>0$ and $\beta\ge 1$
satisfying:

for any $x\in\partial D$ there exists a unit vector $l_x$ such that
\begin{align}
 (l_x,{\bm n})\ge \frac{1}{\beta}
\qquad \mbox{for any}~{\bm n}\in 
\cup_{y\in B(x,\delta)\cap \partial D}{\cal N}_y.\nonumber
\end{align}

\item[{\rm (C)}]
There exists a $C^2_b$ function $f$ on $\RR^d$ 
and a positive constant $\gamma$ such that 
for any $x\in \partial D$, $y\in \bar{D}$, ${\bm n}\in {\cal N}_x$
it holds that
\begin{align}
 \left(y-x,{\bm n}\right)+\frac{1}{\gamma}\left((D
 f)(x),\bm{n}\right)|y-x|^2\ge 0.\nonumber
\end{align}
\end{enumerate}
\end{defin}

We use the following quantities of paths $w_t$
as in \cite{aida-sasaki}.
\begin{align}
\|w\|_{\infty,[s,t]}&=\max_{s\le u\le v\le t}|w_u-w_v|,\\
\|w\|_{[s,t]}&=\sup_{\Delta}\sum_{k=1}^N|w_{t_{k}}-w_{t_{k-1}}|,
\label{total variation}
\end{align}
where $\Delta=\{s=t_0<\cdots<t_N=t\}$ is a partition 
of the interval $[s,t]$.
When the domain $D$ satisfies the conditions (A) and (B),
the Skorohod problem associated with a continuous path
$w\in C([0,T]\to\RR^d)$:
\begin{align}
 \xi_t&=w_t+\phi(t),\quad  \xi_t\in \bar{D}\quad 0\le t\le T,\label{SP1}\\
\phi(t)&=\int_0^t1_{\partial D}(\xi_s){\bm n}(s)d\|\phi\|_{[0,s]},\qquad
{\bm n}(s)\in {\cal N}_{\xi_s}~\mbox{if $\xi_s\in \partial D$}
\label{SP2}
\end{align}
can be uniquely solved.
See \cite{saisho}.
When the mapping $w\mapsto \xi$ is unique, we write
$\Gamma(w)_t=\xi_t$ and $L(w)(t)=\phi(t)$.
The following lemma can be proved
by a similar proof to
that of Lemma~2.3 in \cite{aida-sasaki}.

\begin{lem}\label{estimate on local time}
Assume conditions {\rm (A)} and {\rm (B)} hold.
Let $w_t$ be a continuous path of finite $p$-variation such that
\begin{align*}
 |w_t-w_s|\le \omega(s,t)^{1/p}\qquad 0\le s\le t\le T,
\end{align*}
where $p\ge 1$ and $\omega(s,t)$ is the control function of
$w_t$.
Then the local time term $\phi$ of the solution to the
Skorohod problem associated with $w$ has the following estimate.
\begin{align}
\|\phi\|_{[s,t]}
&\le \beta
\left(\left\{\delta^{-1}G(\|w\|_{\infty,[s,t]})
+1\right\}^{p}\omega(s,t)+1
\right)
\left(G(\|w\|_{\infty,[s,t]})+2\right)\|w\|_{\infty,[s,t]},
\label{estimate on local time 1}
\end{align}
where 
\begin{align*}
G(a)&=4\left\{1+\beta
\exp\left\{\beta\left(2\delta+a\right)/(2r_0)\right\}
\right\}\exp\left\{\beta\left(2\delta+a\right)/(2r_0)\right\},
\quad a\in\RR.
\end{align*}
\end{lem}

\section{Reflected differential equations driven by 
continuous paths of finite $p$-variation with $1\le p<2$
}

Let $x_t$~$(0\le t\le T)$ be a continuous path of finite $p$-variation on
$\RR^n$ with the control function
$\omega(s,t)$, where $1\le p<2$.
We prove the existence of a solution $y_t$ which is also a continuous
path with finite $p$-variation to
the reflected differential equation driven by $x$:
\begin{align}
 y_t=y_0+\int_0^t\sigma(y_s)dx_s+\Phi(t), \qquad y_0\in \bar{D},
\label{rrde for p-variation path}
\end{align}
where $\sigma\in C^1_b(\RR^d, \RR^n\otimes \RR^d)$.
The integral in this equation is a Young integral~\cite{young}.
The following is a main result in this section.
See Remark~\ref{remark on main theorems}.

\begin{thm}\label{main theorem for p-variation path}
Assume that {\rm (A)}  and {\rm (B)} hold.
Then there exists a solution $(y,\Phi)$ to
$(\ref{rrde for p-variation path})$ and satisfies 
\begin{align}
|y_t-y_s|&\le C\omega(s,t)^{1/p} \label{estimate on y}\\
\|\Phi\|_{[s,t]}&\le C\omega(s,t)^{1/p}. \label{estimate on Phi}
\end{align}
Here $C$ is a constant which depends on $\omega(0,T)$
and $\sigma$ and $r_0, \beta, \delta$ in Definition~$\ref{basic definition}$.
\end{thm}

We solve this equation by using the Euler approximation.
Let $\Delta : 0=t_0<t_1<\cdots<t_N=T$ be a partition of
$[0,T]$.
We define $y^{\Delta}$ by the solution to the Skorohod equation:
\begin{align*}
 y^{\Delta}_t&=y^{\Delta}_{t_{k-1}}+\sigma(y^{\Delta}_{t_{k-1}})(x_t-x_{t_{k-1}})+
\Phi^{\Delta}(t)-\Phi^{\Delta}(t_{k-1})\quad
t_{k-1}\le t\le t_k.
\end{align*}
Let
\begin{align}
 I^{\Delta}_{s}(t)&=y^{\Delta}_t-y^{\Delta}_{s}-
\sigma(y^{\Delta}_{s})(x_t-x_{s})-\left(\Phi^{\Delta}(t)-\Phi^{\Delta}(s)\right)
\quad s\le t.\label{Ideltast}
\end{align}
By the definition, we have
$
 I^{\Delta}_{t_k}(t)=0 
$ for all $t_k\le t\le t_{k+1}$
and
For any $s\le t\le u$,
\begin{align}
 I^{\Delta}_{s}(u)-I^{\Delta}_{s}(t)-I^{\Delta}_t(u)&=
\left(\sigma(y^{\Delta}_t)-\sigma(y^{\Delta}_s)\right)(x_u-x_t).\nonumber
\end{align}
Also we write
$\pi^{\Delta}(t)=\max\{t_k~|~t_k\le t, 0\le k\le N\}$
for $0\le t\le T$.

In the following lemma, we use a constant in the estimate
(\ref{estimate on local time 1}).
Let $C_0$ be a positive number such that $C_0>1$ and
\begin{align}
\|\phi\|_{[s,t]}&\le
C_0\left(\omega(s,t)+1\right)\left(e^{C_0\omega(s,t)^{1/p}}+1\right)
\omega(s,t)^{1/p}\label{estimate on local time 2}
\end{align}
holds.
Hence for any positive $\delta$, if $\omega(s,t)$ is sufficiently small, 
$\|\phi\|_{[s,t]}\le (2+\delta)C_0\omega(s,t)^{1/p}$ holds.

\begin{lem}\label{estimate on ydelta lemma 1}
Let $1\le p<\gamma\le 2$.
Let $C_1=3C_0 \|\sigma\|_{\infty}$,
$C_2=1+4C_0\|\sigma\|_{\infty}$ and
$M=\frac{2C_2\|D\sigma\|_{\infty}}{1-2^{1-(\gamma/p)}}$.
For sufficiently small $\ep(\le 1)$
which depends only on $\|\sigma\|_{\infty}, \|D\sigma\|_{\infty}$ 
and $C_0$ such that
for any $t_k\le s\le t$ with
$\omega(t_k,t)\le \ep$,
\begin{align}
 |I^{\Delta}_{t_k}(t)|&\le
M\omega(t_k,t)^{\gamma/p} \label{Itkt}\\
\|\Phi^{\Delta}\|_{[s,t]}&\le C_1
\omega(s,t)^{1/p}.\label{Phitkt}
\end{align}
\end{lem}

\begin{proof}
Note that if (\ref{Itkt}) and (\ref{Phitkt}) hold,
then by taking $\ep$ to be sufficiently small,
we have for $t$ with $\omega(t_k,t)\le \ep$,
\begin{align}
 |y^{\Delta}_t-y^{\Delta}_{t_k}|&\le
\left(M\ep^{(\gamma-1)/p}+3C_0\|\sigma\|_{\infty}+\|\sigma\|_{\infty}\right)
\omega(t_k,t)^{1/p}
\le C_2 \omega(t_k,t)^{1/p}.\label{conclusion of induction}
\end{align}
 Let $K$ be a positive integer.
Consider a claim which depends on $K$:
The estimates (\ref{Itkt}) and (\ref{Phitkt}) hold
for all $t_k$ and $t$,
where $t_k\le t\le t_{k+K}$ and $0\le k\le N-1$.
We prove this claim by an induction on $K$.
Let $K=1$.
Then $I_{t_k}^{\Delta}(t)=0$ for all $t_k\le t\le t_{k+1}$.
Also by taking $\ep$ to be sufficiently small,
\begin{align*}
 \|\Phi^{\Delta}\|_{[s,t]}&\le
3C_0\|\sigma\|_{\infty}\omega(s,t)^{1/p}
\quad \mbox{for $t_k\le s\le t,\,\omega(t_k,t)\le \ep$}.
\end{align*}
Suppose the claim holds for all $K$ which is smaller than or equal to
$K'-1$.
We prove the case $K=K'$.
Let $t_l$ be the largest partition point such that
$t_k\le t_l<t\le t_{k+K'}$ and
$\omega(t_k,t_l)\le \frac{1}{2}\omega(t_k,t)$.
There are two cases, 
(a) $t_l<\pi^{\Delta}(t)$ and
(b) $t_l=\pi^{\Delta}(t)$.
We consider the case (a).
In this case, $t_l<t_{l+1}\le \pi^{\Delta}(t)$.
By the definition,
we have
$\omega(t_k,t_{l+1})\ge\frac{1}{2}\omega(t_k,t)$.
By the superadditivity of $\omega$, we have
\begin{align*}
 \omega(t_{l+1},t)\le \frac{1}{2}\omega(t_k,t).
\end{align*}
We have
\begin{align*}
 |I_{t_k}^{\Delta}(t)|&\le
|I_{t_k}^{\Delta}(t_l)|+|I_{t_l}^{\Delta}(t_{l+1})|+
|I_{t_{l+1}}^{\Delta}(t)|+
|\sigma(y_{t_{l+1}}^{\Delta})-\sigma(y_{t_l}^{\Delta})||x_t-x_{t_{l+1}}|\nonumber\\
&\quad +|\sigma(y_{t_l}^{\Delta})-\sigma(y_{t_k}^{\Delta})|
|x_t-x_{t_l}|\nonumber
\end{align*}
By the assumption of the induction, we have
\begin{align*}
 |I_{t_k}^{\Delta}(t_l)|\le M\omega(t_k,t_l)^{\gamma/p},& \quad
|I_{t_{l+1}}^{\Delta}(t)|\le M\omega(t_{l+1},t)^{\gamma/p}\\
|\sigma(y_{t_{l+1}}^{\Delta})-\sigma(y_{t_l}^{\Delta})||x_t-x_{t_{l+1}}|&\le
C_2\|D\sigma\|_{\infty}\omega(t_l,t_{l+1})^{1/p}
\omega(t_{l+1},t)^{1/p}\\
|\sigma(y_{t_l}^{\Delta})-\sigma(y_{t_k}^{\Delta})||x_t-x_{t_l}|& \le C_2
\|D\sigma\|_{\infty}\omega(t_k,t_l)^{1/p}\omega(t_l,t)^{1/p}
\end{align*}
Therefore
\begin{align*}
 |I_{t_k}^{\Delta}(t)|&\le
M\left(2^{1-(\gamma/p)}+
(1-2^{1-(\gamma/p)})\ep^{(2-\gamma)/p}
\right)\omega(t_k,t)^{\gamma/p}\le M\omega(t_k,t)^{\gamma/p}.
\end{align*}
In the case of (b),
by using the assumption of the induction, we obtain
\begin{align*}
 |I_{t_k}^{\Delta}(t)|&\le
|I_{t_k}^{\Delta}(t_l)|+|I_{t_l}^{\Delta}(t)|+
|\sigma(y_{t_l})-\sigma(y_{t_k})||x_t-x_{t_l}|\nonumber\\
&\le M\omega(t_k,t_l)^{\gamma/p}+
C_2\|D\sigma\|_{\infty}\omega(t_k,t_l)^{1/p}\omega(t_l,t)^{1/p}\nonumber\\
&\le M\left(2^{-\gamma/p}+2^{-1}\left(1-2^{1-(\gamma/p)}\right)\ep^{(2-\gamma)/p}\right)
\omega(t_k,t)^{\gamma/p}\nonumber\\
&\le 2^{-1}M\omega(t_k,t)^{\gamma/p}.
\end{align*}
Next we show $\|\Phi^{\Delta}\|_{[t_k,t]}\le C_1\omega(t_k,t)^{1/p}$
for $t_k,s,t$ with
$t_k\le s\le t\le t_{k+K'}$ and $\omega(t_k,t)\le \ep$.
To this end, we note that 
$\Phi^{\Delta}(t)-\Phi^{\Delta}(t_k)=L(z^{\Delta})(t)$, where
$
z^{\Delta}_t=
I^{\Delta}_{t_k}(t)+y_{t_k}^{\Delta}+\sigma(y^{\Delta}_{t_k})(x_t-x_{t_k})$.
By (\ref{estimate on local time 2}),
it suffices to estimate 
$z^{\Delta}_t$.
Take $s, t$ such that
$t_k\le s<t\le t_{k+K'}$ and
$\omega(t_k,s)\le \ep, \omega(t_k,t)\le \ep$.
We estimate $I_{t_k}^{\Delta}(t)-I_{t_k}^{\Delta}(s)$ by using
\begin{align}
 I^{\Delta}_{t_k}(t)-I^{\Delta}_{t_k}(s)&=
I^{\Delta}_s(t)+\left(\sigma(y^{\Delta}_s)-\sigma(y^{\Delta}_{t_k})
\right)(x_t-x_s).\label{additivity}
\end{align}
Let $t_m$ be the largest number such that $t_m\le s$.
That is $t_m=\pi^{\Delta}(s)$.
Then we have two cases,
(a) $t_k\le t_m\le s<t_{m+1}<t$
and (b) $t_k\le t_m\le s<t\le t_{m+1}$.
First we consider the case (a).
We have
\begin{align*}
 I_s^{\Delta}(t)&=I_{s}^{\Delta}(t_{m+1})+I^{\Delta}_{t_{m+1}}(t)+
(\sigma(y^{\Delta}_{t_{m+1}})-\sigma(y^{\Delta}_s))(x_t-x_{t_{m+1}}).
\end{align*}
Since
$I_{s}^{\Delta}(t_{m+1})=-\left(\sigma(y^{\Delta}_s)-
\sigma(y^{\Delta}_{t_m})\right)(x_{t_{m+1}}-x_s)$,
we have
\begin{align*}
 |I_s^{\Delta}(t_{m+1})|&\le
C_2\|D\sigma\|_{\infty}\omega(t_m,s)^{1/p}\omega(s,t_{m+1})^{1/p}\le
C_2\|D\sigma\|_{\infty}\ep^{1/p}\omega(s,t)^{1/p}.
\end{align*}
By the hypothesis of the induction,
$|I_{t_{m+1}}^{\Delta}(t)|\le M\omega(t_{m+1},t)^{\gamma/p}
\le M\ep^{(\gamma-1)/p}\omega(s,t)^{1/p}$.
Also, 
\begin{align*}
 |(\sigma(y^{\Delta}_{t_{m+1}})-\sigma(y^{\Delta}_s))(x_t-x_{t_{m+1}})|&\le
2C_2\|D\sigma\|_{\infty}\omega(t_m,t_{m+1})^{1/p}\omega(t_{m+1},t)^{1/p}\nonumber\\
&\le 2\ep^{1/p}C_2\|D\sigma\|_{\infty}\omega(s,t)^{1/p}.
\end{align*}
By the assumption of the induction and (\ref{conclusion of induction}),
we have $|y^{\Delta}_s-y^{\Delta}_{t_k}|\le C_2\omega(t_k,s)^{1/p}$.
Hence
\begin{align*}
 |I_{s}^{\Delta}(t)|&\le 
\left(3C_2\|D\sigma\|_{\infty}\ep^{1/p}+
M\ep^{(\gamma-1)/p}\right)
\omega(s,t)^{1/p}\\
|I_{t_k}^{\Delta}(t)-I_{t_k}^{\Delta}(s)|&\le
\left(4C_2\|D\sigma\|_{\infty}\ep^{1/p}+
M\ep^{(\gamma-1)/p}\right)
\omega(s,t)^{1/p}
\end{align*}
and
\begin{align}
|z^{\Delta}_t-z^{\Delta}_s|&\le
\left(4C_2\|D\sigma\|_{\infty}\ep^{1/p}+
M\ep^{(\gamma-1)/p}+\|\sigma\|_{\infty}\right)
\omega(s,t)^{1/p}\label{estimate on zdelta 1}
\end{align}
We consider the case (b).
In this case, $I_{s}^{\Delta}(t)=-\left(\sigma(y^{\Delta}_s)
-\sigma(y^{\Delta}_{t_m})\right)
(x_t-x_s)$.
Noting
$m\le k+K'-1$ and using the assumption of the induction,
we have
\begin{align*}
|y^{\Delta}_s-y^{\Delta}_{t_k}|&\le
|y^{\Delta}_s-y^{\Delta}_{t_m}|+
|y^{\Delta}_{t_m}-y^{\Delta}_{t_k}|\\
&\le C_2\omega(t_m,s)^{1/p}+C_2\omega(t_m,t_k)^{1/p}\\
& \le 2C_2\ep^{1/p}.
\end{align*}
So, we have
\begin{align}
|z^{\Delta}_t-z^{\Delta}_s|
\le \|\sigma\|_{\infty}\omega(s,t)^{1/p}+
3C_2\|D\sigma\|_{\infty}\ep^{1/p}\omega(s,t)^{1/p}.
\label{estimate on zdelta 2}
\end{align}
(\ref{estimate on local time 2}), (\ref{estimate on zdelta 1}) and
 (\ref{estimate on zdelta 2}) implies that
for sufficiently small $\ep$
\begin{align*}
\|\Phi^{\Delta}\|_{[s,t]}&\le
3C_0\|\sigma\|_{\infty}\omega(s,t)^{1/p}
\quad \mbox{for $t_k\le s\le t$ with
$\omega(t_k,s)\le \ep, ~\omega(t_k,t)\le \ep$.
}
\end{align*}

\end{proof}

Actually, the proof of Lemma~\ref{estimate on ydelta lemma 1}
shows

\begin{lem}\label{estimate on ydelta lemma 1b}
Let $\Delta=\{t_k\}$ be a partition of $[0,T]$.
Let $C_1, C_2$ be the same numbers as in 
Lemma~$\ref{estimate on ydelta lemma 1}$.
Then for sufficiently small $0<\ep\le 1$ which depends only on
$\|\sigma\|_{\infty}, \|D\sigma\|_{\infty}$ and $C_0$ such that
for any $s,t$ with $\omega(t_k, s)\le \ep, \omega(t_k,t)\le \ep$
for some $t_k$, we have
\begin{align}
|y^{\Delta}_t-y^{\Delta}_s|&\le C_2\omega(s,t)^{1/p}.
\end{align}
\end{lem}

By Lemma~\ref{estimate on ydelta lemma 1b},
we can prove the following.

\begin{lem}\label{estimate on ydelta}
Let $\ep$ be a positive number in Lemma~$\ref{estimate on ydelta lemma 1b}$.
Let $\Delta=\{t_k\}_{k=0}^N$ be a partition of
$[0,T]$ such that 
\begin{align}
 \sup_{0\le k\le l\le N-1}|\omega(t_k,t_{l+1})-\omega(t_k,t_l)|\le \ep/2.
\label{assumption on the partition}
\end{align}
Then there exists $C>0$ such that
for any $0\le s\le t\le T$ the following estimates hold.
The constant $C$ depends only on
$\sigma$, $p$ and $D$.
\begin{enumerate}
\item[$(1)$] $|y^{\Delta}_t-y^{\Delta}_s|\le
	     C\left(1+\omega(0,T)\right)\omega(s,t)^{1/p}$
\item[$(2)$] $\|\Phi^{\Delta}\|_{[s,t]}\le C\left(1+\omega(0,T)\right)
\omega(s,t)^{1/p}$.
\end{enumerate}
\end{lem}

\begin{proof}
[Proof of Lemma~$\ref{estimate on ydelta}$] 
We note that the statement is true if
$t_k\le s\le t\le t_{k+1}$ for some $k$
by Lemma~\ref{estimate on local time}.
Let us consider general cases.
We define a subsequence $\{s_k\}_{k=0}^{N'}$ of $\{t_k\}$ in the following way.
Let $s_0=t_0=0$.
When $s_k$ is defined, we define
$s_{k+1}$ is the smallest $t_i$ such that
$\omega(s_k,t_i)>\ep/2$ and $t_i>s_k$.
If there does not exist such a $t_{i}$, we set
$s_{k+1}=t_N$.
By the assumption (\ref{assumption on the partition}),
$\omega(s_k,s_{k+1})\le \max\left(\omega(s_k,t_{i-1})+\ep/2, \ep/2\right)\le \ep$.
Hence by Lemma~\ref{estimate on ydelta lemma 1b},
\begin{align*}
 |y^{\Delta}_{s_{k+1}}-y^{\Delta}_{s_k}|
\le C_2\omega(s_k,s_{k+1})^{1/p}.
\end{align*}
By the superadditivity of $\omega$, we have
\begin{align*}
 \omega(0,T)\ge\sum_{k=0}^{N'-1}\omega(s_k,s_{k+1})\ge
(N'-1)\ep/2
\end{align*}
which implies
$N'\le 1+2\omega(0,T)/\ep$.
For $0\le s<t\le T$, let
us choose the numbers $l, m$ so that
$s_l\le s<s_{l+1}\le s_m\le t< s_{m+1}$.
Then
\begin{align*}
 |y^{\Delta}_t-y^{\Delta}_s|&\le
|y^{\Delta}_t-y^{\Delta}_{s_m}|
+\sum_{k=l+1}^{m-1}|y^{\Delta}_{s_{k+1}}-y^{\Delta}_{s_k}|
+|y^{\Delta}(s_{l+1})-y^{\Delta}(s)|\nonumber\\
&\le C_2\omega(s_m,t)^{1/p}+\sum_{k=l+1}^{m-1}C_2\omega(s_k,s_{k+1})^{1/p}+
C_2\omega(s,s_{l+1})^{1/p}\nonumber\\
&\le C_2\left(2\omega(0,T)/\ep+3\right)\omega(s,t)^{1/p}.
\end{align*}
For $\Phi^{\Delta}$,  similarly, we have
\begin{align*}
\|\Phi^{\Delta}\|_{[s,t]}&=
\|\Phi^{\Delta}\|_{[s_m,t]}+
\sum_{k=l+1}^{m-1}\|\Phi^{\Delta}\|_{[s_k,s_{k+1}]}+
\|\Phi^{\Delta}\|_{[s,s_{l+1}]}\nonumber\\
&\le C_1\omega(s_m,t)^{1/p}+\sum_{k=l+1}^{m-1}C_1\omega(s_k,s_{k+1})^{1/p}+
C_1\omega(s,s_{l+1})^{1/p}\nonumber\\
&\le C_1\left(2\omega(0,T)/\ep+3\right)\omega(s,t)^{1/p}.
\end{align*}
These estimates complete the proof.
\end{proof}

\begin{proof}[Proof of Theorem~$\ref{main theorem for p-variation path}$]
Let us consider a sequence of partitions
$\Delta(n)=\{t(n)_k\}$ of $[0,T]$ such that
\begin{itemize}
 \item[(a)] the estimate $(\ref{assumption on the partition})$ holds for all $\Delta(n)$,
\item[(b)] $\lim_{n\to\infty}\max_{k\ge 0}|t(n)_{k+1}-t(n)_k|=0$.
\end{itemize}
These partitions exist because the mapping $(s,t)\mapsto \omega(s,t)$ is continuous.
By Lemma~\ref{estimate on ydelta}, there exists a subsequence
$y^{\Delta(n_k)}$ and $\Phi^{\Delta(n_k)}$ converge uniformly to 
continuous paths
$y^{\infty}$ and $\Phi^{\infty}$ respectively 
which also satisfy (\ref{estimate on y}) and (\ref{estimate on Phi}).
Then these subsequences converge in
$p'$-variation norm for any $p'>p$.
The solution $y^{\Delta(n_k)}$ satisfies
\begin{align*}
 y^{\Delta(n_k)}_t=y_0+\int_0^t
\sigma\left(y^{\Delta(n_k)}(\pi^{\Delta}(u))\right)dx_u+
\Phi^{\Delta(n_k)}(t).
\end{align*}
By taking the limit $n_k\to\infty$ and by the 
continuity theorem of Young integral and the continuity of the Skorohod map,
we see that $(y^{\infty},\Phi^{\infty})$ is a solution 
to the equation.
\end{proof}

Before closing this section, we make a simple remark on the 
continuity of the solution map $x\mapsto y$ when
$x$ is a bounded variation path.

\begin{rem}\label{remark on continuity}
{\rm
Let 
$x_t$ and $x'_t$ be continuous bounded variation paths
on $\RR^n$ starting at $0$.
Let $D$ be the domain which satisfies {\rm (A), (B), (C)}.
Let us consider two reflected ODEs and their solutions $y_t,y_t'$:
\begin{align*}
 y_t&=y_0+\int_0^t\sigma(y_s)dx_s+\Phi(t)\\
y'_t&=y_0+\int_0^t\sigma(y'_s)dx'_s+\Phi'(t)
\end{align*}
Let $m_t=|y_t-y'_t|^2e^{-\frac{2}{\gamma}\left(f(y_t)+f(y'_t)\right)}$.
Then by calculating $dm_t$ as in
\cite{lions-sznitman,saisho},
and by the Gronwall inequality, we obtain
\begin{align*}
 \sup_{0\le s\le t}|y_s-y'_s|&
\le Ce^{C'\left(\|x\|_{[0,t]}+\|x'\|_{[0,t]}\right)}
\|x-x'\|_{[0,t]}.
\end{align*}
This implies the solution map $x\mapsto y$ is a
Lipschitz continuous map between the set of bounded variation
 paths
and the set of continuous paths.
}
\end{rem}

\section{Reflected differential equations driven by 
$p$-rough path with $2\le p<3$
}

In this section, we prove the existence of a solution to
reflected differential equations driven by
rough path.
We mainly follow the formulation of rough path 
in~\cite{lyons98, lq, davie}.
See also
\cite{coutin-qian, friz-victoir0, friz-victoir, gubinelli}.
First, we define reflected differential equation driven by
rough path.

\begin{defin}\label{definition of rrde}
Let $D$ be a connected domain in $\RR^d$ for which 
the condition {\rm (A)} holds.
Let $2\le p<3$.
Let $X_{s,t}=(1,X^1_{s,t},X^2_{s,t})\in \Omega_p(\RR^n)$~$(0\le s\le t\le T)$
be a $p$-rough path.
Let $Y_{s,t}=(1,Y^1_{s,t},Y^2_{s,t})\in \Omega_p(\RR^d)$ 
be a $p$-rough path and
$\Phi(t)$ $(0\le t\le T)$ be a continuous bounded variation path
on $\RR^d$.
Let $\sigma\in C^2_b(\RR^d, \RR^n\otimes\RR^d)$.
The pair $(Y,\Phi)$ is called a solution to a
rough differential equation on $D$ driven by $X$ with normal reflection 
with the starting point $y_0\in \bar{D}$:
\begin{align}
  dY_t&=\sigma(Y_t)dX_t+d\Phi(t)\quad 0\le t\le T, \quad Y_0=y_0,\label{rrde}
\end{align}
if the following hold.
\begin{enumerate}
\item[$(1)$] 
Let $Y_t=y_0+Y^1_{0,t}$.
Then $Y_t\in \bar{D}$~$(0\le t\le T)$ and
it holds that
there exists a Borel measurable map
$s(\in [0,T])\mapsto \bm{n}(s)\in \RR^d$
such that 
$\bm{n}(s)\in {\mathcal N}_{Y_s}$ if
$Y_s\in \partial D$ and
\begin{align}
 \Phi(t)&=\int_0^t1_{\partial D}(Y_s)\bm{n}(s)d\|\Phi\|_{[0,s]}
\quad 0\le t\le T.\label{local time equation}
\end{align}

\item[$(2)$] 
$Y_{s,t}$ is a solution to the following rough differential equation.
\begin{align}
 dY_t&=\hat{\sigma}(Y_t)d\hat{X}_t\quad 0\le t\le T, \quad Y_0=y_0,\label{rde}
\end{align}
where 
$\hat{\sigma}(x)$ is a linear mapping from
$\RR^n\oplus \RR^d$ to $\RR^d$ defined by
$\hat{\sigma}(x)(\xi,\eta)=\sigma(x)\xi+\eta$
and
the driving rough path $\hat{X}\in \Omega_p(\RR^n\oplus \RR^d)$
is given by
\begin{align*}
 \hat{X}_{s,t}^1&=(X^1_{s,t},\Phi(t)-\Phi(s))\\
\hat{X}^2_{s,t}&=\Big(X^2_{s,t}, \int_s^tX^1_{s,u}\otimes d\Phi(u),
\int_s^t\left(\Phi(u)-\Phi(s)\right)\otimes dX^1_{s,u},
\int_s^t\left(\Phi(u)-\Phi(s)\right)\otimes d\Phi(u)\Big).
\end{align*}
\end{enumerate}
\end{defin}

Note that if $X_{s,t}$ is a rough path defined by a continuous path
$X_t$ of finite $q$-variation with $1\le q<2$, then the solution 
$Y_t$ coincides with the solution in the sense of Section 2.
Below, we assume $\sigma\in C^2_b$.
To solve this equation, 
we consider the Euler approximation modifying the Davies'
approximation for rough differential equations without reflection terms.
Let $\Delta : 0=t_0<t_1<\cdots<t_N=T$ be a partition of
$[0,T]$.
Let us consider a Skorohod problem :
\begin{align}
y^{\Delta} _t&=y^{\Delta}_{t_{k-1}}+
\sigma(y^{\Delta}_{t_{k-1}})(x_{t}-x_{t_{k-1}})+
(D\sigma)(y^{\Delta}_{{t_{k-1}}})(\sigma(y^{\Delta}_{t_{k-1}})X^2_{t_{k-1},t})
\nonumber\\
&\quad 
+(D\sigma)(y^{\Delta}_{t_{k-1}})\left(\int_{t_{k-1}}^t
(\Phi^{\Delta}(r)-\Phi^{\Delta}(t_{k-1}))\otimes dx_r\right)
+\Phi^{\Delta}(t)-\Phi^{\Delta}(t_{k-1})\nonumber\\
&\quad\qquad y^{\Delta}_t\in \bar{D}, \quad
y^{\Delta}_0=y_0, \quad t_{k-1}\le t\le t_{k},
\quad 1\le k\le N, \label{implicit skorohod equation}
\end{align}
where $x_t=X^1_{0,t}$.
That is, the pair $(y^{\Delta}_t, \Phi^{\Delta}_t-\Phi^{\Delta}_{t_{k-1}})$
is the solution to the Skorohod problem associated with
the continuous path
\begin{align*}
& y^{\Delta}_{t_{k-1}}+
\sigma(y^{\Delta}_{t_{k-1}})(x_{t}-x_{t_{k-1}})+
(D\sigma)(y^{\Delta}_{{t_{k-1}}})(\sigma(y^{\Delta}_{t_{k-1}})X^2_{t_{k-1},t})
\\
&\quad 
+(D\sigma)(y^{\Delta}_{t_{k-1}})\left(\int_{t_{k-1}}^t
(\Phi^{\Delta}(r)-\Phi^{\Delta}(t_{k-1}))\otimes dx_r\right)
\qquad t_{k-1}\le t\le t_k.
\end{align*}
Since this is an implicit Skorohod problem,
the existence of the solution 
is not trivial.
In view of this,
we consider the following condition (D) and assumptions 
(H1) and (H2) on
$D$.

\begin{assumption}
\begin{enumerate}
\item[{\rm (D)}]
Condition {\rm (A)} is satisfied and there exist constants
$K_1\ge 0$ and $0<K_2<r_0$ such that
\begin{align*}
 |\bar{x}-\bar{y}|\le
(1+K_1\ep)|x-y|
\end{align*}
holds for any $x,y\in \RR^d$ with
$|x-\bar{x}|\le K_2$, $|y-\bar{y}|\le K_2$,
where $\ep=\max\{|x-\bar{x}|, |y-\bar{y}|\}$.
Here $\bar{x}$ denotes the nearest point of $x$ in $\bar{D}$.

 \item[{\rm (H1)}] The condition {\rm (A)} holds and
the Skorohod problem 
$(\ref{SP1})$, $(\ref{SP2})$
is uniquely solved for any $w$.
Moreover, there exists a positive constant $C_D$
such that for all continuous paths $w$ on $\RR^d$
\begin{align*}
 \|L(w)\|_{[s,t]}&\le C_D\|w\|_{\infty,[s,t]} \quad 0\le s\le t\le T.
\end{align*}

\item[{\rm (H2)}] 
The condition {\rm (A)} holds and
the Skorohod problem $(\ref{SP1})$, $(\ref{SP2})$
is uniquely solved for any $w$.
Moreover, there exists a positive constant $C_D'$ such that
for all continuous paths $w,w'$ on $\RR^d$ 
\begin{align*}
 \|L(w)-L(w')\|_{\infty,[0,t]}\le
C_D'\left\{\|w-w'\|_{[0,t]}+|w(0)-w'(0)|\right\}.
\end{align*}
\end{enumerate} 
\end{assumption}

\begin{rem}{\rm
It is proved in \cite{tanaka} that
the condition (H1) holds 
if 
$D$ is convex and
there exists a unit vector $l\in \RR^d$ such that
$$
\inf\{(l,{\bm n}(x))~|~{\bm n}(x)\in {\cal N}_x, x\in \partial D\}>0.
$$
The condition (H2) holds if the conditions (B) and 
(D) are satisfied.
This is due to \cite{saisho}.}
\end{rem}

About the existence and uniqueness of solutions
to (\ref{implicit skorohod equation}),
we have the following.

\begin{lem}
Let $\eta_t$ be a continuous path on $\RR^d$ with $\eta_0=0$ and
$x_t$ be a continuous path of finite $p$-variation on $\RR^n$ with $x_0=0$
for some $p\ge 1$.
Let $F$ be a linear
mapping from $\RR^d\otimes \RR^n$ to $\RR^d$.
We consider the following implicit Skorohod equation:
\begin{align}
 y_t&=y_0+\eta_t+F\left(\int_0^t\Phi(r)\otimes dx_r\right)+\Phi(t)
\qquad y_0\in \bar{D}\quad 0\le t\le T,
\label{euler approximation of rrde0}
\end{align}
where $y_t\in \bar{D}$~$(0\le t\le T)$
and $\Phi(t)$ is a continuous bounded variation path
which satisfies 
$$
L\left(y_0+\eta_{\cdot}+F\left(\int_0^{\cdot}\Phi(r)\otimes
 dx_r\right)\right)(t)
=\Phi_t\quad 0\le t\le T,\qquad \Phi_0=0.
$$
\begin{enumerate}
 \item[$(1)$] Assume {\rm (H2)} are satisfied and
$x_t$ is bounded variation.
Then there exists a unique solution $(y_t, \Phi(t))$
to $(\ref{euler approximation of rrde0})$.
\item[$(2)$] Assume {\rm (H1)} holds.
There exists a solution $(y_t,\Phi(t))$
to $(\ref{euler approximation of rrde0})$.
\end{enumerate}
\end{lem}

\begin{proof}
(1)~
By (H2),
we see the unique existence of $\Phi$,
by a standard iteration procedure on continuous path spaces with the
norm $\|~\|_{\infty,[0,T]}$
considering the equation in the small interval, if necessary.
This arguments produce the solution for the whole interval
$[0,T]$.

\noindent
(2)
First we prove the existence of a solution
on a small interval $[0,T']$, where $T'<T$.
We specify $T'$ later.
Let $\Delta=\{t_k\}_{k=0}^N$ be a partition of $[0,T']$.
We consider the Euler approximation of $y$.
\begin{align*}
y^{\Delta}_t&=y^{\Delta}_{t_k}+\eta_t-\eta_{t_k}+F\left(\Phi^{\Delta}(t_k)\otimes
 (x_t-x_{t_k})\right)
+\Phi^{\Delta}(t)-\Phi^{\Delta}(t_k)
\quad t_k\le t\le t_{k+1}.
\end{align*}
That is, $y^{\Delta}, \Phi^{\Delta}$ satisfies
\begin{align*}
 y^{\Delta}_t&=y_0+\eta_t+
F\left(\int_0^t\Phi^{\Delta}
(\pi^{\Delta}(r))\otimes dx_r\right)+\Phi^{\Delta}(t)
\qquad ~ 0\le t\le T'.
\end{align*}
Let $0\le s<t\le T'$.
If $t_{k-1}\le s<t\le t_k$ for some $k$, then
\begin{align}
 \int_s^t\Phi^{\Delta}(\pi^{\Delta}(r))\otimes dx_r&=
\Phi^{\Delta}(t_{k-1})\otimes (x_t-x_s).\label{Phideltast1}
\end{align}
We consider the case
where $0\le t_{k-1}\le s<t_k<\cdots<t_l\le t<t_{l+1}\le T'$.
Then 
\begin{align}
\int_s^t 
\Phi^{\Delta}(\pi^{\Delta}(r))\otimes dx_r&=
\Phi^{\Delta}(t_{k-1})\otimes(x_{t_k}-x_s)+
\Phi^{\Delta}(t_l)\otimes(x_t-x_{t_l})\nonumber\\
&\quad +\Phi^{\Delta}(t_{l-1})\otimes x_{t_l}-
\Phi^{\Delta}(t_k)\otimes x_{t_k}\nonumber\\
&\quad +
\sum_{m=k}^{l-2}\left(
\Phi^{\Delta}(t_m)-\Phi^{\Delta}(t_{m+1})\right)\otimes x_{t_{m+1}}.
\label{Phideltast2}
\end{align}
Therefore we have for all $0\le s<t\le T'$,
\begin{align}
\left|\int_s^t\Phi^{\Delta}(\pi^{\Delta}(r))\otimes dx_r\right| 
&\le
3\|\Phi^{\Delta}\|_{[0,T']}\|x\|_{\infty,[s,t]}+
2\|\Phi^{\Delta}\|_{[s,t]}\|x\|_{\infty,[0,T']}\label{estimate on integral}\\
&\le
5\|\Phi^{\Delta}\|_{[0,T']}\|x\|_{\infty,[0,T']}.\nonumber
\end{align}
Hence by (H1),
\begin{align*}
\|\Phi^{\Delta}\|_{[0,T']}&\le
C_D\left(\|\eta\|_{\infty,[0,T']}+
5\|F\|\|x\|_{\infty,[0,T']}\|\Phi^{\Delta}\|_{[0,T']}
\right).
\end{align*}
Therefore if $\|x\|_{\infty,[0,T']}\le 1/\left(10C_D\|F\|\right)$,
\begin{align}
 \|\Phi^{\Delta}\|_{[0,T']}&\le 2C_D\|\eta\|_{\infty, [0,T']}.
\label{estimate on Phidelta2}
\end{align}
Substituting this into (\ref{estimate on integral}),
we obtain for any $0\le s\le t\le T'$,
\begin{align}
\left|\int_s^t 
\Phi^{\Delta}(\pi^{\Delta}(r))\otimes dx_r\right|&\le
6C_D\|\eta\|_{\infty,[0,T']}\|x\|_{\infty,[s,t]}
+2\|\Phi^{\Delta}\|_{[s,t]}\|x\|_{\infty,[0,T']}.
 \end{align}
Hence, again by applying (H1), we obtain
\begin{align}
 \|\Phi^{\Delta}\|_{[s,t]}&\le
C_D\|\eta\|_{\infty,[s,t]}+
6C_D^2\|F\|\,\|\eta\|_{\infty,[0,T']}\|x\|_{\infty,[s,t]}
+2C_D\|F\|\, \|\Phi^{\Delta}\|_{[s,t]}\|x\|_{\infty,[0,T']}.
\end{align}
Consequently,
if 
\begin{align}
 \|x\|_{\infty,[0,T']}\le(10C_D\|F\|)^{-1}
\label{small time interval}
\end{align}
then
\begin{align*}
 \|\Phi^{\Delta}\|_{[s,t]}&\le\frac{5}{4}
C_D\|\eta\|_{\infty,[s,t]}+
10C_D^2\|F\|\,\|\eta\|_{\infty,[0,T']}\|x\|_{\infty,[s,t]}
\qquad 0\le s\le t\le T'.
\end{align*}
Now we choose $T'$ so that 
(\ref{small time interval}) holds.
Then $\{\Phi^{\Delta}\}_{\Delta}$ is a family of
equicontinuous and bounded functions on $[0,T']$
and so there exists a sequence $|\Delta_n|\to 0$ such that
$\Phi^{\Delta_n}$ converges to a certain $\Phi$ uniformly on $[0,T']$.
By the estimate (\ref{estimate on Phidelta2}),
this convergence takes place for all $p$-variation norm ($p>1$) on 
$[0,T']$.
Therefore $F\left(\int_0^{t}\Phi^{\Delta_n}
(\pi^{\Delta_n}(r))\otimes dx_r\right)$ 
converges uniformly to
$F\left(\int_0^t\Phi(r)\otimes dx_r\right)$.
Here we use the property of Young integrals.
Also $y^{\Delta_n}_t$ converges uniformly.
We denote the limit by $y$.
Then $(y_t,\Phi(t))$~$(0\le t\le T')$ is a 
solution to (\ref{euler approximation of rrde0}).
Next, we need to construct a solution after time $T'$.
For $t\ge T'$, 
(\ref{euler approximation of rrde0}) reads
\begin{align}
 y_t&=y_{T'}+
(\eta_t-\eta_{T'})+F\left(\Phi_{T'}\otimes(x_t-x_{T'})\right)\nonumber\\
&\quad +F\left(\int_{T'}^t
(\Phi(r)-\Phi(T'))\otimes dx_r\right)+
\Phi_t-\Phi_{T'}.
\end{align}
Since $T'$ depends only on $C_D$ and $\|F\|$,
by iterating the above procedure, we can get a solution defined on 
$[0,T]$.
\end{proof}

By the above lemma,
we see that there exist a solution  $(y^{\Delta}, \Phi^{\Delta})$
to the implicit Skorohod equation (\ref{implicit skorohod equation}).
Using this approximation solution, we can prove the existence of 
a solution to reflected rough differential equations.
Now we state our main theorem in this section.

\begin{thm}\label{main theorem for p-rough path}
 Assume {\rm (H1)} and $\sigma\in C^3_b$.
Let $\omega$ be the control function of $X_{s,t}$, {\it i.e.},
it holds that
\begin{align*}
|X^i_{s,t}|&\le \omega(s,t)^{i/p}\qquad 0\le s\le t\le T, \quad i=1,2.
\end{align*}
Then there exists a solution $(Y,\Phi)$ to the reflected rough differential equation
$(\ref{rrde})$
such that
for all $0\le s\le t\le T$,
\begin{align}
 |Y^i_{s,t}|&\le C(1+\omega(0,T))^3\omega(s,t)^{i/p}, \qquad i=1,2,
 \label{estimate on Yi}\\
\|\Phi\|_{[s,t]}&\le C(1+\omega(0,T))^3\omega(s,t)^{1/p} ,
\label{estimate on Phi2}
\end{align}
where the positive constant $C$ depends only on
$\sigma, C_D, p$.
\end{thm}

In the proof of this theorem, we use Lyons' continuity theorem.
That is why we assume $\sigma\in C^3_b$.
However, it may not be necessary.
Actually $\sigma\in C^2_b$ is sufficient for the proof of
Lemma~\ref{estimate on ydelta lemma 2} and 
Lemma~\ref{estimate on ydelta2}.
Here we make remarks on this theorem together with
Theorem~\ref{main theorem for p-variation path}.

\begin{rem}\label{remark on main theorems}
{\rm
(1)
At the moment, I do not prove the uniqueness yet and it is not clear
to see whether the functional $X\mapsto \Phi$,
$X\mapsto Y$ is continuous or not.
Actually, at the moment, I do not know the existence of Borel measurable
selection of the mapping.
We consider this measurable selection problem 
for geometric rough path at the end of this section.
If there are no boundary terms, the functional
$X\mapsto Y$ is continuous and this is known as Lyons' continuity
theorem and universal limit theorem.
If the continuity theorem would hold, then by applying it to the case
of Brownian rough path,
it would imply the
strong convergence of Wong-Zakai approximation which
was proved in \cite{aida-sasaki} under general conditions 
on the boundary.
We discuss the relation between the solution to reflected
rough differential equation driven by Brownian rough path and
the solution to reflected SDE driven by Brownian motion
later.

\noindent
(2) We consider the case where $D$ is a half space.
In this simplest case too, we have difficulties 
to prove the uniqueness of solutions and
 continuity theorems with respect to driving paths (rough paths)
in the equations (\ref{rrde for p-variation path}) and (\ref{rrde}).
We explain the reason.
When $D$ is a half space, the Skorohod mapping $\Gamma$ is given explicitly
and it is globally Lipschitz continuous in the set of
continuous path spaces with the sup-norm.
This nice result is used in the studies \cite{anderson-orey,
 doss}.
However, it is not Lipschitz continuous in the $\lambda$-H\"older continuous path
spaces $C^{\lambda}$.
This is pointed out by
Ferrante and Rovira~\cite{ferrante-rovira}
who studied reflected differential equations driven by
H\"older continuous paths on half spaces.
This implies the difficulty of the study of the uniqueness of
solutions to reflected differential equations as pointed out in their
paper.
We may need to restrict the set of
solutions to reflected rough differential equations
to obtain the uniqueness.
In the usual rough differential equations, we have locally
Lipschitz continuities of the solutions 
with respect to the driving rough paths.
On the other hand, it is not difficult to show that 
$\Gamma$ is H\"older continuous mapping in $C^{\lambda}$.
Hence it may be possible to prove such a weaker
continuity of the solution mapping for reflected rough 
differential equation.
}
\end{rem}

To prove this theorem, we argue similarly to the
case $1\le p<2$.
When $\Phi^{\Delta}(t)$ is defined,
let
\begin{align*}
J^{\Delta}_{s}(t)
&=
I^{\Delta}_{s}(t)-
D\sigma(y^{\Delta}_{s})(\sigma(y^{\Delta}_{s}))(X^2_{s,t})\nonumber\\
&\quad -
(D\sigma)(y^{\Delta}_{s})\left(\int_{s}^t
\left(\Phi^{\Delta}(r)-\Phi^{\Delta}(s)\right)
\otimes dx_r\right)\quad s\le t.
\end{align*}
The definition of $I^{\Delta}_s(t)$ is similar to
(\ref{Ideltast}) just replacing $\Phi^{\Delta}$ by a solution to
(\ref{implicit skorohod equation}).
By the definition of $y^{\Delta}$, we have
$J_{t_k}^{\Delta}(t)=0$ for $t_k\le t\le t_{k+1}$.
We define
$J^{\Delta}(s,t,u)=J^{\Delta}_s(u)
-J^{\Delta}_s(t)-J^{\Delta}_t(u)$.
By an easy calculation, 
we have for $s\le t\le u$,
\begin{align*}
J^{\Delta}(s,t,u)
&=
\Bigl(\sigma(y^{\Delta}_t)-
\sigma(y^{\Delta}_s)-(D\sigma)(y^{\Delta}_s)(y^{\Delta}_t-y^{\Delta}_s)
+(D\sigma)(y^{\Delta}_s)(I^{\Delta}_s(t))\Bigr)(x_t-x_u)\nonumber\\
&\quad +\left((D\sigma)(y^{\Delta}_t)(\sigma(y^{\Delta}_t))-
(D\sigma)(y^{\Delta}_s)(\sigma(y^{\Delta}_s))\right)(X^2_{t,u})\nonumber\\
& \quad+\left((D\sigma)(y^{\Delta}_t)-
(D\sigma)(y^{\Delta}_s)\right)\left(\int_t^u
\left(\Phi^{\Delta}(r)-\Phi^{\Delta}(t)\right)\otimes dx_r\right).
\end{align*}
This relation plays important role as in \cite{davie} and the proof in
Lemma~\ref{estimate on ydelta lemma 1}
in the calculation below.

\begin{lem}\label{estimate on ydelta lemma 2}
Suppose {\rm (H1)} hold.
Let $2\le p<\gamma\le 3$.
There exist positive constants $M$ and $\ep$ which depend only on
$\sigma$ and $C_D$ such that
if $\omega(t_k,t)\le \ep(\le 1)$ and $t_k\le s\le t$, then
\begin{align}
|J_{t_k}^{\Delta}(t)|&\le M\omega(t_k,t)^{\gamma/p}\label{Jtkt}\\
\|\Phi^{\Delta}\|_{[s,t]}&\le C_3\omega(s,t)^{1/p},\label{Phitkt 2}
\end{align}
where $C_3=2C_D\|\sigma\|_{\infty}$.
The constant $M$ is specified in $(\ref{M})$.
\end{lem}

\begin{proof}
 If (\ref{Jtkt}) and (\ref{Phitkt 2}) hold,
then
\begin{align*}
 |y^{\Delta}_t-y^{\Delta}_{t_k}|&\le
\left(M\ep^{(\gamma-1)/p}+\|\sigma\|_{\infty}+C_3+
\|D\sigma\|_{\infty} \|\sigma\|_{\infty} \ep^{1/p}+
2C_3\|D\sigma\|_{\infty} \ep^{1/p}\right)\omega(t_{k},t)^{1/p}\nonumber\\
&\le C_4\omega(t_{k},t)^{1/p},
\end{align*}
where
$C_4=1+C_3+\|\sigma\|_{\infty}$
and we have used the relation
\begin{align*}
\int_{t_k}^t\left(\Phi^{\Delta}(r)-\Phi^{\Delta}(t_k)\right)
\otimes dx_r&=
\left(\Phi^{\Delta}(t)-\Phi^{\Delta}(t_k)\right)
\otimes(x_{t}-x_{t_k})-\int_{t_k}^td\Phi^{\Delta}(r)\otimes(x_r-x_{t_{k}}).
\end{align*}
Also
\begin{align*}
 |I_{t_k}^{\Delta}(t)|&\le
M\omega(t_k,t)^{\gamma/p}+
\|D\sigma\|_{\infty}\|\sigma\|_{\infty}\omega(t_k,t)^{2/p}+
2C_3\|D\sigma\|_{\infty}\omega(t_{k},t)^{2/p}\nonumber\\
&\le C_5\omega(t_k,t)^{2/p},
\end{align*}
where
$C_5=1+2C_3\|D\sigma\|_{\infty}+\|D\sigma\|_{\infty}\|\sigma\|_{\infty}$.
Let
\begin{align*}
z^{\Delta}_t&=
 y^{\Delta}_{t_{k}}+
\sigma(y^{\Delta}_{t_{k}})(x_{t}-x_{t_{k}})+
(D\sigma)(y^{\Delta}_{{t_{k}}})(\sigma(y^{\Delta}_{t_{k}})X^2_{t_{k},t})
\nonumber\\
&\quad 
+(D\sigma)(y^{\Delta}_{t_{k}})\left(\int_{t_{k}}^t
(\Phi^{\Delta}(r)-\Phi^{\Delta}(t_{k}))\otimes dx_r\right)
+J^{\Delta}_{t_k}(t)\qquad t\ge t_k.
\end{align*}
Then $\Phi^{\Delta}(t)-\Phi^{\Delta}(t_{k})=L(z^{\Delta})(t)$
for
$t\ge t_k$.
We use this relation to estimate $\Phi^{\Delta}$.
Let $K$ be a positive integer.
Consider a claim which depends on $K$:
The estimates (\ref{Jtkt}) and (\ref{Phitkt 2}) hold
for all $t_k$ and $t$,
where $t_k\le t\le t_{k+K}$ and $0\le k\le N-1$.
We prove this claim by an induction on $K$.
Let $K=1$.
By the definition,
$J_{t_{k}}^{\Delta}(t)=0$ for any $t_k\le t\le t_{k+1}$.
We estimate the bounded variation norm of $\Phi^{\Delta}$.
Let $t_k\le s\le t\le t_{k+1}$.
Noting Chen's identity
\begin{align}
X^2_{t_k,t}-X^2_{t_k,s}
&=X^2_{s,t}+(x_s-x_{t_k})\otimes (x_t-x_s).\label{chen}
\end{align}
and
by (H1),
\begin{align*}
 \|\Phi^{\Delta}\|_{[s,t]}&\le
C_D\left(\|\sigma\|_{\infty}+2\|D\sigma\|_{\infty}\|\sigma\|_{\infty}\ep^{1/p}\right)
\omega(s,t)^{1/p}
+
C_D\ep^{1/p}\|D\sigma\|_{\infty}\|\Phi^{\Delta}\|_{[s,t]}\nonumber\\
&\quad +C_D\|D\sigma\|_{\infty}\|\Phi^{\Delta}\|_{[t_k,t]}\omega(s,t)^{1/p}
\end{align*}
which implies for sufficiently small $\ep$,
\begin{align*}
 \|\Phi^{\Delta}\|_{[s,t]}&\le 2C_D\|\sigma\|_{\infty}\omega(s,t)^{1/p}.
\end{align*}
Suppose the claim holds for all $K$ which is smaller than or equal to
$K'-1$.
We prove the case $K=K'$.
Let $t_l$ be the largest partition point such that
$t_k\le t_l<t\le t_{k+K'}$ and
$\omega(t_k,t_l)\le \frac{1}{2}\omega(t_k,t)$.
There are two cases, 
(a) $t_l<\pi^{\Delta}(t)$ and
(b) $t_l=\pi^{\Delta}(t)$.
We consider the case (a).
In this case, $t_l<t_{l+1}\le \pi^{\Delta}(t)$.
By the definition,
we have
$\omega(t_k,t_{l+1})\ge\frac{1}{2}\omega(t_k,t)$.
By the superadditivity of $\omega$, we have
\begin{align}
 \omega(t_{l+1},t)\le \frac{1}{2}\omega(t_k,t).
\end{align}
We have
\begin{align*}
 |J_{t_k}^{\Delta}(t)|&\le
|J_{t_k}^{\Delta}(t_l)|+|J_{t_l}^{\Delta}(t_{l+1})|+
|J_{t_{l+1}}^{\Delta}(t)|\nonumber\\
& +
|J^{\Delta}(t_k,t_l,t)|+|J^{\Delta}(t_l,t_{l+1},t)|
\end{align*}
By the assumption of the induction and the choice of $t_l$,
\begin{align*}
 |J_{t_k}^{\Delta}(t_l)|\le
2^{-\gamma/p}M\omega(t_k,t)^{\gamma/p},\quad
|J_{t_{l+1}}^{\Delta}(t)|\le 2^{-\gamma/p}M\omega(t_k,t)^{\gamma/p}.
\end{align*}
By the assumption of the induction,
we have 
\begin{align*}
|J^{\Delta}(t_k,t_l,t)|&\le
(C_4/2)\|D^2\sigma\|_{\infty}\omega(t_k,t_l)^{2/p}\omega(t_l,t)^{1/p}
+C_5\|D\sigma\|_{\infty}\omega(t_k,t_l)^{2/p}\omega(t_l,t)^{1/p}\nonumber\\
&\quad +
C_4\left(\|D^2\sigma\|_{\infty}\|\sigma\|_{\infty}+\|D\sigma\|_{\infty}^2\right)
\omega(t_k,t_l)^{1/p}\omega(t_l,t)^{2/p}\nonumber\\
&\quad+
2C_4C_3\|D^2\sigma\|_{\infty}\omega(t_k,t_l)^{1/p}
\omega(t_l,t)^{2/p}.
\end{align*}
Here we have used that
if
$t_k<t_l$ we can use the assumption of the induction and so, 
\begin{align*}
\left|\int_{t_l}^t\left(\Phi^{\Delta}(r)-\Phi^{\Delta}(t_l)\right)
\otimes dx_r\right|&=
\left|\left(\Phi^{\Delta}(t)-\Phi^{\Delta}(t_l)\right)
\otimes(x_{t}-x_{t_l})-\int_{t_l}^td\Phi^{\Delta}(r)\otimes(x_r-x_{t_{l}})\right|
\nonumber\\
&=
2\|\Phi^{\Delta}\|_{[t_l,t]}\omega(t_l,t)^{1/p}\nonumber\\
&\le
2C_3\omega(t_l,t)^{2/p}.
\end{align*}
Similarly,
\begin{align*}
|J^{\Delta}(t_l,t_{l+1},t)|&\le
(C_4^2/2)\|D^2\sigma\|_{\infty}\omega(t_l,t_{l+1})^{2/p}
\omega(t_{l+1},t)^{1/p}
+C_5\|D\sigma\|_{\infty}\omega(t_l,t_{l+1})^{2/p}\omega(t_{l+1},t)^{1/p}\nonumber\\
&\quad +
C_4\left(\|D^2\sigma\|_{\infty}\|\sigma\|_{\infty}+\|D\sigma\|_{\infty}^2\right)
\omega(t_l,t_{l+1})^{1/p}
\omega(t_{l+1},t)^{2/p}\nonumber\\
&\quad+
2C_3C_4\|D^2\sigma\|_{\infty}\omega(t_{l},t_{l+1})^{1/p}
\omega(t_{l+1},t)^{2/p}.
\end{align*}
Consequently,
\begin{align*}
|J^{\Delta}_{t_k}(t)|&\le
2^{1-(\gamma/p)}M\omega(t_k,t)^{\gamma/p}
+
\ep^{(3-\gamma)/p}C_6\omega(t_k,t)^{\gamma/p},
\end{align*}
where
\begin{align*}
 C_6&=
C_4^2
\|D^2\sigma\|_{\infty}+2C_5\|D\sigma\|_{\infty}+
2C_4\left(\|D^2\sigma\|_{\infty}\|\sigma\|_{\infty}+\|D\sigma\|_{\infty}^2
\right)+4C_3C_4\|D^2\sigma\|_{\infty}.
\end{align*}
Therefore, 
if $M$ satisfies
\begin{align}
 M\ge \frac{C_6}{1-2^{1-(\gamma/p)}},\label{M}
\end{align}
then the desired estimate for $J_{t_k}^{\Delta}(t)$ holds.
In the case of (b),
by using the assumption of the induction and noting
$J_{t_l}^{\Delta}(t)=0$, we obtain
\begin{align*}
 |J_{t_k}^{\Delta}(t)|&\le
|J_{t_k}^{\Delta}(t_l)|+|J_{t_l}^{\Delta}(t)|+
|J^{\Delta}(t_k,t_l,t)|\nonumber\\
&\le M\omega(t_k,t_l)^{\gamma/p}+
|J^{\Delta}(t_k,t_l,t)|\nonumber\\
&\le 2^{-\gamma/p}M\omega(t_k,t)^{\gamma/p}+
\left(\ep^{(3-\gamma)/p}/2\right)C_6\omega(t_k,t)^{\gamma/p}.
\end{align*}
Hence, under the condition (\ref{M}), the desired estimate for
$J_{t_k}^{\Delta}(t)$ holds.
We show $\|\Phi^{\Delta}\|_{[s,t]}\le C_3\omega(s,t)^{1/p}$
for $t_k\le s<t\le t_{k+K'}$ with $\omega(t_k,t)\le \ep$.
We have
\begin{align*}
 J^{\Delta}_{t_k}(t)-J^{\Delta}_{t_k}(s)&=
J^{\Delta}_s(t)+J^{\Delta}(t_k,s,t).
\end{align*}
Let $t_m$ be the largest number such that $t_m\le s$.
Then we have two cases,
(a) $t_k\le t_m\le s<t_{m+1}<t$
and (b) $t_k\le t_m\le s<t\le t_{m+1}$.
We consider the case (a).
We can apply the assumption of the induction to $t_k, s$ and
we obtain,
\begin{align*}
 |J^{\Delta}(t_k,s,t)|&\le
2^{-1}C_4^2\|D^2\sigma\|_{\infty}C_4^2\omega(t_k,s)^{2/p}\omega(s,t)^{1/p}\nonumber\\
& \quad+
C_5\|D\sigma\|_{\infty}\omega(t_k,s)^{2/p}\omega(s,t)^{1/p}\nonumber\\
& +C_4\left(\|D^2\sigma\|_{\infty}\|\sigma\|_{\infty}+\|D\sigma\|_{\infty}^2\right)
\omega(t_k,s)^{1/p}\omega(s,t)^{2/p}\nonumber\\
&\quad +
2C_4\|D^2\sigma\|_{\infty}\omega(t_k,s)^{1/p}\|\Phi^{\Delta}\|_{[s,t]}\omega(s,t)^{1/p}.
\end{align*}
We have
\begin{align*}
 J_s^{\Delta}(t)&=
J_s^{\Delta}(t_{m+1})+J_{t_{m+1}}^{\Delta}(t)+
J^{\Delta}(s,t_{m+1},t).
\end{align*}
Since $J_s^{\Delta}(t_{m+1})=-J^{\Delta}(t_m,s,t_{m+1})$,
\begin{align*}
 |J_s^{\Delta}(t_{m+1})|&\le
2^{-1}\|D^2\sigma\|_{\infty}\omega(t_m,s)^{2/p}\omega(s,t_{m+1})^{1/p}
+C_5\|D\sigma\|_{\infty}\omega(t_m,s)^{2/p}\omega(s,t_{m+1})^{1/p}\nonumber\\
&\quad +
C_4\left(\|D^2\sigma\|_{\infty}\|\sigma\|_{\infty}+\|D\sigma\|_{\infty}^2\right)
\omega(t_m,s)^{1/p}\omega(s,t_{m+1})^{2/p}\nonumber\\
&\quad+
2C_4\|D^2\sigma\|_{\infty}\omega(t_m,s)^{1/p}
\|\Phi^{\Delta}\|_{[s,t_{m+1}]}\omega(s,t_{m+1})^{1/p}.
\end{align*}
Note that
\begin{align*}
|I_s^{\Delta}(t_{m+1})|&\le
|I_{t_m}^{\Delta}(t_{m+1})|+|I_{t_m}^{\Delta}(s)|+
|\sigma(y^{\Delta}_{s})-\sigma(y^{\Delta}_{t_m})|\cdot |x_{t_{m+1}}-x_s|\nonumber\\
&\le
C_5\omega(t_m,t_{m+1})^{2/p}+C_5\omega(t_m,s)^{2/p}+
\|D\sigma\|_{\infty}C_4\omega(t_m,s)^{1/p}\omega(s,t_{m+1})^{1/p}\nonumber\\
&\le (2C_5+C_4\|D\sigma\|_{\infty})\ep^{1/p}\omega(t_k,t)^{1/p}\\
|y_{s}^{\Delta}-y_{t_{m+1}}^{\Delta}|&\le
2C_4\omega(t_m,t_{m+1})^{1/p}.
\end{align*}
Hence,
\begin{align*}
 |J^{\Delta}(s,t_{m+1},t)|&\le
2C_4^2\|D^2\sigma\|_{\infty}\omega(t_m,t_{m+1})^{2/p}\omega(t_{m+1},t)^{1/p}
\nonumber\\
&\quad
+\|D\sigma\|_{\infty}(2C_5+C_4\|D\sigma\|_{\infty})\ep^{1/p}\omega(t_k,t)^{1/p}
\omega(t_{m+1},t)^{1/p}\nonumber\\
&\quad +
2C_4\left(\|D^2\sigma\|_{\infty}\|\sigma\|_{\infty}+\|D\sigma\|_{\infty}^2\right)
\omega(t_m,t_{m+1})^{1/p}\omega(t_{m+1},t)^{2/p}\nonumber\\
&\quad+
4C_4\|D^2\sigma\|_{\infty}\omega(t_m,t_{m+1})^{1/p}
\|\Phi^{\Delta}\|_{[t_{m+1},t]}\omega(t_{m+1},t)^{1/p}.
\end{align*}
By the assumption of induction,
\begin{align*}
|J_{t_{m+1}}^{\Delta}(t)|&\le
M\omega(t_{m+1},t)^{\gamma/p}.
\end{align*}
Because
\begin{align*}
 \int_{s}^t\left(\Phi^{\Delta}(r)-\Phi^{\Delta}(t_k)\right)\otimes
 dx_r&=
\left(\Phi^{\Delta}(t)-\Phi^{\Delta}(t_k)\right)\otimes(x_t-x_s)
-\int_s^td\Phi^{\Delta}(r)\otimes (x_r-x_s),
\end{align*}
we have
\begin{align*}
 \left|\int_{s}^t\left(\Phi^{\Delta}(r)-\Phi^{\Delta}(t_k)\right)\otimes
 dx_r\right|&\le
\|\Phi^{\Delta}\|_{[t_k,t]}\omega(s,t)^{1/p}
+\|\Phi^{\Delta}\|_{[s,t]}\omega(s,t)^{1/p}\nonumber\\
&\le C_3\ep^{1/p}\omega(s,t)^{1/p}+\ep^{1/p}\|\Phi^{\Delta}\|_{[s,t]}.
\end{align*}
By Chen's identity and
putting the estimates above together,
by (H1), for sufficiently small $\ep$, we have 
\begin{align}
\|\Phi^{\Delta}\|_{[s,t]}
&\le C_D\|z^{\Delta}\|_{\infty,[s,t]}\nonumber\\
&\le
C_D\|\sigma\|_{\infty}\omega(s,t)^{1/p}
+\ep^{1/p}C_D
\left(2\|D\sigma\|_{\infty}\|\sigma\|_{\infty}+C_3\|D\sigma\|_{\infty}\right)
\omega(s,t)^{1/p}\nonumber\\
&\quad + C_D\|D\sigma\|_{\infty}(C_6\ep^{2/p}+2\ep^{1/p})
\|\Phi^{\Delta}\|_{[s,t]}\nonumber\\
&\quad +C_DC_7\ep^{1/p}\omega(s,t)^{1/p},\label{estimate on Phidelta}
\end{align}
where $C_7$ depends only on $p,\sigma, D$.
Therefore, for sufficiently small $\ep$, we obtain
$\|\Phi^{\Delta}\|_{[s,t]}\le C_3\omega(s,t)^{1/p}$.
We consider the case (b).
Since $I_{t_k}^{\Delta}(s)=I_{t_k}^{\Delta}(t_m)+I^{\Delta}_{t_m}(s)
+\left(\sigma(y_{t_m})-\sigma(y_{t_k})\right)(x_s-x_{t_m})$,
by using the assumption of the induction, we have
\begin{align*}
 |I_{t_k}^{\Delta}(s)|&\le
C_5\omega(t_k,t_m)^{2/p}+C_5\omega(t_m,s)^{2/p}+
C_4\|D\sigma\|_{\infty}\omega(t_k,t_m)^{1/p}\omega(t_m,s)^{1/p}\nonumber\\
&\le (2C_5+C_4\|D\sigma\|_{\infty})\ep^{2/p}.
\end{align*}
Since $J_s^{\Delta}(t)=-J^{\Delta}(t_m,s,t)$, we have
\begin{align*}
 |J^{\Delta}_s(t)|&\le
2^{-1}C_4^2\|D^2\sigma\|_{\infty}\omega(t_m,s)^{2/p}\omega(s,t)^{1/p}
+C_5\|D\sigma\|_{\infty}\omega(t_m,s)^{2/p}\omega(s,t)^{1/p}\nonumber\\
&\quad
 +C_4\left(\|D^2\sigma\|_{\infty}\|\sigma\|_{\infty}+\|D\sigma\|_{\infty}^2\right)
\omega(t_m,s)^{1/p}\omega(s,t)^{2/p}\nonumber\\
&\quad +2C_4\|D^2\sigma\|_{\infty}\omega(t_m,s)^{1/p}\omega(s,t)^{1/p}
\|\Phi^{\Delta}\|_{[s,t]}.
\end{align*}
Therefore, by the same argument as the case (a), we complete the proof of
the case (b) and the proof of the lemma is finished.
\end{proof}

Actually, the above proof shows stronger estimates
similarly to the case $1\le p<2$.
For $t_k\le s\le t$ with $\omega(t_k,t)\le\ep$
\begin{align*}
 |J^{\Delta}_{t_k}(t)-J^{\Delta}_{t_k}(s)|\le
C_8\ep^{1/p}\omega(s,t)^{1/p}.
\end{align*}
Thus taking smaller $\ep$ if necessary,
we have
\begin{align*}
 |y^{\Delta}_t-y^{\Delta}_s|&\le
C_4\omega(s,t)^{1/p}
\quad
\mbox{for $t_k\le s\le t$ with $\omega(t_k,t)\le \ep$}.
\end{align*}

For general $s,t$, we have the following estimates.

\begin{lem}\label{estimate on ydelta2}
Let $\ep$ be a positive number specified in the above argument.
 Let $\Delta=\{t_k\}_{k=1}^N$ be a partition of
$[0,T]$ which satisfies $(\ref{assumption on the partition})$.
Then
there exists $C>0$ such that
for any $0\le s\le t\le T$ the following estimates hold.
The constant $C$ depends only on
$\sigma$, $p$ and $D$.
\begin{enumerate}
\item[$(1)$] $|y^{\Delta}_t-y^{\Delta}_s|\le C
\left(1+\omega(0,T)\right)\omega(s,t)^{1/p}$
\item[$(2)$] $\|\Phi^{\Delta}\|_{[s,t]}\le C
\left(1+\omega(0,T)\right)\omega(s,t)^{1/p}$.
\end{enumerate}
\end{lem}

\begin{proof}
 The proof of this lemma is similar to
that of Lemma~\ref{estimate on ydelta}.
\end{proof}

Now we prove our main theorem.

\begin{proof}[Proof of Theorem~$\ref{main theorem for p-rough path}$]
Let $\widehat{X^{\Delta}}$ be the naturally defined 
$p$-rough path associated with the $p$-rough path
$X$ and $\Phi^{\Delta}$ as in Definition~\ref{definition of rrde} (2).
Thanks to the above lemma,
this family of $p$-rough path has a common control function
$C \omega$ for some positive constant $C$ which is independent of
$\Delta$.
Let $p'>p$.
Since the two-parameter function
$(s,t)\mapsto \widehat{X^{\Delta}}_{s,t}$ and
$y^{\Delta}_t$
are equicontinuos
(we need Chen's identity to prove the equicontinuity of the former),
there exist subseqeunces
$\widehat{X^{\Delta_n}}, y^{\Delta_n}$,
a $p$-rough path $\hat{X}\in\Omega_p(\RR^n\oplus \RR^d)$,
a continuous path $y$ and
a positive decreasing sequence $\delta_n\downarrow 0$ such that
\begin{align*}
 \left|
\widehat{X^{\Delta_n}}_{s,t}-\hat{X}_{s,t}
\right|&\le \delta_n \omega(s,t)^{1/p'} \quad 0\le s\le t\le T,\\
\lim_{n\to\infty}\max_{0\le t\le T}|y^{\Delta_n}_t-y_t|&=0,
\end{align*}
where 
$\Delta_{n+1}$ is a subdivision of $\Delta_n$ and
$|\Delta_n|\to 0$.
We denote the limit of $\Phi^{\Delta_n}(t)$ by $\Phi(t)$.
Clearly, the estimate (\ref{estimate on Phi2}) holds for this $\Phi$.
The limit $\hat{X}$ is naturally defined rough path by
$X$ and $\Phi$ as in Definition~\ref{definition of rrde} (2).
Also we have for all $0\le s\le t\le T$,
\begin{align}
& \Biggl|y_t-y_s-
\sigma(y_s)(x_{t}-x_{s})-(\Phi(t)-\Phi(s))\nonumber\\
&\qquad -
(D\sigma)(y_{s})(\sigma(y_{s})X^2_{s,t})
-(D\sigma)(y_{s})\left(\int_{s}^t
(\Phi(r)-\Phi(s))\otimes dx_r\right)\Biggr|
\le
C\omega(s,t)^{\gamma/p}.\label{solution in the sense of Davie}
\end{align}
(\ref{Jtkt}) shows (\ref{solution in the sense of Davie})
for $s=t_k\in \cup_n \Delta_n$ with
$\omega(t_k,t)\le \ep$.
By the denseness of $\cup_n\Delta_n$ and the continuity
of the functions on the both sides in (\ref{solution in the sense of
 Davie}), we see 
that this estimate holds for $s,t$ with $\omega(s,t)\le \ep$.
When $\omega(s,t)\ge \ep$, the estimate clearly holds.
This shows $y_t$ is a solution to
\begin{align*}
 dy_t=\hat{\sigma}(y_t)d\hat{X}_t
\end{align*}
in the sense of Davie~\cite{davie}.
Also we can find a $p$-rough path $Y_{s,t}\in \Omega_p(\RR^d)$
so that $y_t=y_0+Y^1_{0,t}$
and the equation (\ref{rrde}) is satisfied.
We refer the reader for the construction of $Y_{s,t}$ to \cite{davie}.
We write $Y_t=y_0+Y^1_{0,t}$.
Since $\hat{X}_{s,t}$ has the control function
$C(1+\omega(0,T))\omega(s,t)$, the estimate on the rough differential
equations implies the estimate (\ref{estimate on Yi}).
We have to show $Y_t$ and $\Phi(t)$ is the solution to the
Skorohod problem associated with
the first level path 
$
y_0+\int_0^t\sigma(Y_s)dX^1_s.
$
To this end, we consider the solution 
$Y^{\Delta_n}_{s,t}$ associated with 
$\widehat{X^{\Delta_n}}_{s,t}$.
Let $Y^{\Delta_n}_t=y_0+(Y^{\Delta_n})^1_{0,t}$.
Since $\widehat{X^{\Delta_n}}$ converges to $\hat{X}$ in
$\Omega_{p'}(\RR^n\oplus \RR^d)$,
by Lyons' continuity theorem of solutions to rough differential
equations, we have
$
\lim_{n\to\infty}\|Y^{\Delta_n}-y\|_{\infty,[0,T]}=0.
$
Hence
\begin{align*}
\lim_{n\to\infty}\|Y^{\Delta_n}-y^{\Delta_n}\|_{\infty,[0,T]}=0.
\end{align*}
Also by Lyons' continuity theorem for the integrals of $p'$-rough path,
\begin{align*}
 \lim_{n\to\infty}\left\|
\int_0^{\cdot}\sigma(Y^{\Delta_n}_s)dX^1_s
-\int_0^{\cdot}\sigma(Y_s)dX^1_s\right\|_{\infty,[0,T]}=0.
\end{align*}
Let $z^{\Delta_n}_t=y^{\Delta_n}_t-\Phi^{\Delta_n}_t$.
Then $(y^{\Delta_n}_t,\Phi^{\Delta_n}_t)$ is the solution
to the Skorohod problem associated with $z^{\Delta_n}_t$.
Since
$z^{\Delta_n}_t=y_t^{\Delta_n}-Y^{\Delta_n}_t+y_0+
\int_0^t\sigma(Y^{\Delta_n}_s)dX^1_s$,
$z^{\Delta_n}_t$
converges to $y_0+\int_0^t\sigma(Y_s)dX^1_s$ uniformly.
By the continuity of the Skorohod map (see \cite{saisho}),
this shows the desired result.
\end{proof} 

Let $G\Omega_p(\RR^n)$ be the set of geometric rough paths
with $2\le p<3$.
That is, this set is the closure of the set of smooth rough paths
in $p$-variation norm.
For solutions to reflected rough differential equations
driven by geometric rough path, we can prove the existence of
measurable selection of the solution mapping.

\begin{thm}
Assume $D$ satisfies {\rm (H1)} and $\sigma\in C^3_b$.
There exists a universally measurable map
$I : G\Omega_p(\RR^n)\to G\Omega_p(\RR^d)\times V_p(\RR^d)$
such that the following hold.
Here $V_p(\RR^d)$ denotes the set of continuous paths
of finite
$p$-variation in $\RR^d$ defined on $[0,T]$.

\noindent
$(1)$ 
For any $X\in G\Omega_p(\RR^n)$, 
$I(X)$ is a solution
to $(\ref{rrde})$ and satisfies the estimates
$(\ref{estimate on Yi})$ and $(\ref{estimate on Phi2})$.

\noindent
$(2)$
There exists a sequence of smooth rough paths
$\{X_N\}\subset G\Omega_p(\RR^n)$ such that
$\lim_{N\to\infty}X_N=X$ and
$
\lim_{N\to\infty}I(X_N)=I(X),
$
where the convergences take place in the topology
$G\Omega_{p}(\RR^n)$ and the
product
topology of $G\Omega_{p'}(\RR^d)\times V_{p}(\RR^d)$
for all $p<p'<3$
respectively.
\end{thm}

\begin{proof}
 For any $X\in G\Omega_p(\RR^n)$, there exists a sequence
of smooth rough paths $\{X_N\}$ such that
$\lim_{N\to\infty}\|X_N-X\|_p=0$.
Let $(Y_N,\Phi_N)$ be the solution to reflecting rough differential
 equation driven by
$X_N$.
Let $\widetilde{X_N}\in G\Omega_p(\RR^n\oplus\RR^d\oplus\RR^d)$ 
be the smooth rough path associated with
$(X_N,\Phi_N,Y_N)$ similarly to
Definition~\ref{definition of rrde} (2).
By the estimate (\ref{estimate on Yi}) and (\ref{estimate on
 Phi2}),
there exists a subsequence
$\widetilde{X_{N_k}}$ which converges to an element
$\widetilde{X}$ in
$G\Omega_{p}(\RR^n\oplus \RR^d\oplus \RR^d)$
in the topology
of $G\Omega_{p'}(\RR^n\oplus\RR^d\oplus \RR^d)$
for any $p<p'<3$.
A pair of the rough path and the bounded variation path 
$\pi(\tilde{X})=(Y,\Phi)\in G\Omega_p(\RR^d)\otimes V_p(\RR^d)$
which is obtained by a projection of
$\widetilde{X}$ is a solution to 
(\ref{rrde}) driven by
$X$.
This follows from the Lyon's continuity theorem and the continuity
of the Skorohod map.
Let $\Theta\subset G\Omega_p(\RR^n\oplus \RR^d\oplus\RR^d)\times G\Omega_p(\RR^n)$ 
be the set consisting of all limit points
$(\widetilde{X},X)$.
Then clearly $\Theta$ is a closed subset.
Hence by the measurable selection theorem, there exists
a universally measurable map
${\cal I} : G\Omega_p(\RR^n)\to G\Omega_p(\RR^n\oplus \RR^d\oplus\RR^d)$
such that $\{({\cal I}(X),X)~|~X\in G\Omega_p(\RR^n)\}\subset \Theta$.
The mapping $I(X)=\pi({\cal I}(X))$ is the desired map.
\end{proof}

\section{Reflected stochastic differential equation}

In this section, we consider stronger topology than $p$-variation topology
of geometric rough path.
The set of geometric rough paths
$G\Omega_p(\RR^n)$ is the closure of the set of
smooth rough paths defined by continuous bounded variation paths
with respect to the distance $d_p$ below and consists
$X_{s,t}=(1,X^1_{s,t}, X^2_{s,t})$ where
$X^1_{s,t}, X^2_{s,t}$
are $\RR^n$ and $\RR^n\otimes \RR^n$-valued continuous maps
satisfying
Chen's identity and
\begin{align}
\sup_{0\le s<t\le T}\frac{|X^i_{s,t}|}{|t-s|^{i/p}}& <\infty.
\end{align}
The distance is given by
\begin{align*}
 d_p(X,X')&=
\sum_{i=1}^2\sup_{0\le s< t\le
 T}\frac{|X^i_{s,t}-(X^{'})^i_{s,t}|}{|t-s|^{i/p}},
\qquad X, X'\in G\Omega_p(\RR^n).
\end{align*}
$\left(G\Omega_p(\RR^n),d_p\right)$ is 
a complete separable metric space.
Let $W^n=C([0,T]\to\RR^n~|~B(0)=0)$
be the classical Wiener space.
That is, $W^n$ is a probability space with the Wiener measure $\mu$.
The coordinate process $t\mapsto B(t)$
is a realization of Brownian motion.
Let
\begin{align*}
 B^N(t)=B(t^N_{k-1})+\frac{B(t^N_{k})-B(t^N_{k-1})}{\Delta_N}
(t-t^N_{k-1})
\quad t^N_{k-1}\le t\le t^N_{k},
\end{align*}
where $t^N_k=kT/(2^N)$~$(1\le k\le 2^N)$,
$\Delta_N=2^{-N}T$ and $\Delta_kB^N=B(t^N_{k})-B(t^N_{k-1})$.
We may omit superscript $N$ in the notation 
$t^N_k$.
Consider a smooth rough path $B^N_{s,t}$ over $B^N$.
Then we can see that there exists a subset
$\Omega\subset W^n$ such that $\mu(\Omega)=1$ and any $B\in \Omega$
satisfies
that $B^N_{s,t}$ converges in $G\Omega_p(\RR^n)$.
The limit which is denoted by $B_{s,t}$ is called a
Brownian rough path.
We can take control function $\omega(s,t)$ such that
$\omega(s,t)=
C(X)(t-s)$, where
$C(X)=\left(d_p(0,X)^{p}+d_p(0,X)^{p/2}\right)$
and $X_{s,t}=B_{s,t}, B^N_{s,t}$.
It is not difficult to see that
\begin{align}
E[C(B)^q]+\sup_NE[C(B^N)^q]<\infty \quad \mbox{for any $q\ge 1$}.
\label{moment estimate for CB}
\end{align}
Now, again, we assume $\sigma\in C^2_b(\RR^d,\RR^n\otimes \RR^d)$ 
throughout this section.
Let $Y^N$ be the solution to reflected ODE on a domain $D\subset \RR^d$:
\begin{align}
dY^N(t)&=\sigma(Y^N(t))dB^N(t)+d\Phi^N(t), \quad Y^N(0)=y_0.
\label{reflected ode}
\end{align}
Under the assumption (H1),
by Theorem~\ref{main theorem for p-rough path}
in the Section 4,
we have
\begin{align}
|Y^N(B)^i_{s,t}|&\le
g(d_p(0,B^N))(t-s)^{i/p}\quad i=1,2,\label{brp 3}\\
\|\Phi^N(B)\|_{[s,t]}&\le 
g(d_p(0,B^N))(t-s)^{1/p}, \label{brp 4}
\end{align}
where $g$ is a polynomial function.
Therefore, by the same reasoning as in the proof of
Theorem~\ref{main theorem for p-rough path}, for any $B\in \Omega$,
there exists a subsequence $N_k(B)\uparrow+\infty$ such that
$Y^{N_k(B)}(B)_{s,t}$ and
$\Phi^{N_k(B)}(B)(t)$ converge in the topology of $p'$-rough path and
$p'$-variation path respectively.
The limit is a solution to reflected rough differential equation driven
by $B_{s,t}$.
However, we cannot conclude that the limit and the solution is unique
by this argument.
On the other hand, the solution $Y^N$ is the Wong-Zakai approximation of
$Y^S(t)$ which is the solution to the reflected SDE driven by Brownian motion:
\begin{align*}
 dY^S(t)&=\sigma(Y^S(t))\circ dB(t)+d\Phi^S(t),\quad Y^S(0)=y_0,
\end{align*}
where $\circ dB(t)$ denotes the Stratonovich integral
and $\Phi^S$ is the local time term.
We use the notation $Y^S$ to distinguish the solution in the sense of
It\^o calculus from the solution in the sense of rough path.
Note that in \cite{aida-sasaki}, we used the notation
$X^N(t)$ for the Wong-Zakai approximation.
Let us consider the case $D=\RR^{d}$.
Then Lyon's continuity theorem
and the coincidence of the solution in the sense of
It\^o's SDE and rough differential equations,
imply that the Wong-Zakai approximation of
the solution converges to the solution in the sense of
It\^o calculus uniformly.
However, we cannot do such a thing if $\partial D\ne \emptyset$
because we do not prove the continuity theorem yet.
In \cite{aida-sasaki}, we proved that
$Y^N(t)$ converges to $Y^S(t)$ uniformly on $[0,T]$
for almost all $B$.
By the results in \cite{aida-sasaki},
we can prove the following lemma.

\begin{lem}\label{refinement}
Assume conditions {\rm (A), (B), (C)} are satisfied for $D$.
Then for any $\ep>0$, there exists a positive constant
$C_{\ep}(T)$ independent of $N$ such that
\begin{align*}
E\left[
\max_{0\le t\le T}\left|\int_0^t\sigma(Y^N(s))dB^N(s)-\int_0^t\sigma(Y^S(s))\circ
 dB(s)\right|^2\right]
&\le C_{\ep}(T)\cdot 2^{-(1-\ep)N/6}.
\end{align*}
\end{lem}

Thanks to the lemma above, applying the Borel-Cantelli lemma, 
we see that there exists a full measure subset 
$\Omega'\subset W^n$ such that
\begin{align*}
 \max_{0\le t\le T}\left|\int_0^t\sigma(Y^N(s))dB^N(s)-
\int_0^t\sigma(Y^S(s))\circ dB(s)\right|\to 0
\quad \mbox{for all $B\in \Omega'$}.
\end{align*}
Hence by the continuity property of
the Skorohod mapping, $\Phi^N(t)$ also converges to
$\Phi^S(t)$ uniformly for all $B\in \Omega'$.
Therefore, $Y^N(B)_{s,t}$ 
converges to a certain $p$-rough path $Y(B)_{s,t}$
for all $B\in \Omega'\cap \Omega$, without taking subsequences.
The pair $(Y(B)_{s,t},\Phi^S(t))$
is a solution to rough differential equation
driven by $B\in \Omega\cap \Omega'$
and $Y^S(t)=y_0+Y(B)^1_{0,t}$.
Also (\ref{brp 3}) and (\ref{brp 4}) imply
the following.

\begin{thm}
Assume {\rm (A), (B), (C), (H1)}.
Then we have for $0\le s\le t\le T$,
\begin{align*}
 |Y^S(t)-Y^S(s)|&\le C \left(1+d_p(0,B)\right)^3|t-s|^{1/p},\\
\|\Phi^S\|_{[s,t]}&\le C \left(1+d_p(0,B)\right)^3|t-s|^{1/p},
\end{align*}
where $C$ is a positive constant
which depends only on 
$\sigma, D, p$.
\end{thm}

\begin{proof}[Proof of Lemma~$\ref{refinement}$]
In this proof, we use the estimate obtained in \cite{aida-sasaki}.
Note that some notation there are different from
those in this paper.
Take points such that $t_l<t\le t_{l+1}$.
We have
\begin{align*}
\left|\int_0^t\sigma(Y^N(s))dB^N(s)-\int_0^{t_l}\sigma(Y^N(s))dB^N(s)\right|
&\le
C|\Delta_lB^N|.
\end{align*}
Hence 
\begin{align*}
\lefteqn{E\left[
\max_{0\le s\le t}\left|\int_0^s\sigma(Y^N(u))dB^N(u)-
\int_0^s\sigma(Y^S(u))\circ
 dB(u)\right|^2\right]}\nonumber\\
&\le
3E\left[
\max_{0\le k\le l}\left|\int_0^{t_k}\sigma(Y^N(s))dB^N(s)-\int_0^{t_k}
\sigma(Y^S(s))\circ
 dB(s)\right|^2\right]\nonumber\\
&\quad +3CE\left[\max_{k}|\Delta_kB^N|^{2}\right]+
3E\left[\max_{|u-v|\le 2^{-N}T, 0\le u\le v\le T}
\left|\int_u^v\sigma(Y^S(s))\circ dB(s)\right|^{2}\right]\nonumber\\
&\le
3E\left[
\max_{0\le k\le l}\left|\int_0^{t_k}\sigma(Y^N(s))dB^N(s)-\int_0^{t_k}
\sigma(Y^S(s))\circ
 dB(s)\right|^2\right]
+C_{\ep}\left(2^{-N}T\right)^{1-\ep},
\end{align*}
where $\ep$ is any positive number.
Let $\pi^N(t)=\max\{t^N_k~|~t^N_k\le t\}$.
 \begin{align*}
 \lefteqn{\int_0^{t_l}\sigma(Y^N(s))dB^N(s)-\int_0^{t_l}
\sigma(Y^S(s))\circ
 dB(s)}\nonumber\\
&=
\int_0^{t_l}\left(\sigma(Y^N(\pi^N(s)))-\sigma(Y^S(s))\right)dB(s)\nonumber\\
&\quad +\Biggl\{
\sum_{k=1}^l
\int_{t_{k-1}}^{t_k}\int_{t_{k-1}}^s
\left(D\sigma\right)(Y^N(u))\left(\sigma(Y^N(u))\frac{\Delta_k
  B^N}{\Delta_N}du\right)
\left(\frac{\Delta_k B^N}{\Delta_N}\right)ds\nonumber\\
&\quad
-\int_0^{t_l}\frac{1}{2}\tr\left(D\sigma\right)(Y^S(s))
(\sigma(Y^S(s)))ds\Biggr\}\nonumber\\
&\quad +
\sum_{k=1}^l
\int_{t_{k-1}}^{t_k}\int_{t_{k-1}}^s
\left(D\sigma\right)(Y^N(u))\left(\sigma(Y^N(u))d\Phi^N(u)\right)
\left(\frac{\Delta_k B^N}{\Delta_N}\right)ds
\nonumber\\
&=:I_1^N(t_l)+I^N_2(t_l)+I^N_3(t_l).
 \end{align*}
Noting
\begin{align*}
I^N_1(t)&=
\int_0^t
\Bigl(\sigma(Y^N(\pi^N(s)))-\sigma(Y^S(\pi^N(s)))\Bigr)dB(s)\nonumber\\
&\quad +
\int_0^t\Bigl(\sigma(Y^S(\pi^N(s)))-\sigma(Y^S(s))\Bigr)dB(s),
\end{align*}
and by using Burkholder-Davis-Gundy's inequality and
estimates in Theorem 2.9 and Lemma~4.5 in \cite{aida-sasaki},
we obtain
\begin{align*}
E\left[\max_{0\le s\le t}|I^N_1(s)|^{2}\right]&\le
Ct\left(2^{-N}T\right)^{(1-\ep)/6}
+C\cdot 2^{-N}tT.
\end{align*}

\begin{align*}
 I_2^N(t_l)&=
\frac{1}{2}\int_0^{t_{l}}
\left(\tr\left(D\sigma\right)(Y^N(\pi^N(s)))(\sigma(Y^N(\pi^N(s))))
-\tr\left(D\sigma\right)(Y^S(s))(\sigma(Y^S(s)))\right)
ds\nonumber\\
&\quad +
\sum_{k=1}^l
\int_{t_{k-1}}^{t_k}\int_{t_{k-1}}^s
\Biggl\{\left(D\sigma\right)(Y^N(u))\left(\sigma(Y^N(u))\frac{\Delta_k
  B^N}{\Delta_N}\right)\nonumber\\
&\quad -\left(D\sigma\right)(Y^N(\pi^N(u)))\left(\sigma(Y^N(\pi^N(u)))\frac{\Delta_k
  B^N}{\Delta_N}\right)\Biggr\}du\,
\left(\frac{\Delta_k B^N}{\Delta_N}\right)
ds\nonumber\\
&\quad +
\frac{1}{2}\sum_{k=1}^l
\Biggl\{\left(D\sigma\right)(Y^N(t_{k-1}))
\left(\sigma(Y^N(t_{k-1}))\Delta_k
  B^N\right)\left(\Delta_kB^N\right)\nonumber\\
&\quad-
\sum_{i=1}^n
\left(D\sigma\right)(Y^N(t_{k-1}))\left(\sigma(Y^N(t_{k-1}))e_i\right)
\left(e_i\right)2^{-N}T\Biggr\},\nonumber\\
&=:I^N_{2,1}(t_l)+I^N_{2,2}(t_l)+I^N_{2,3}(t_l)
\end{align*}
where $e_i$ is a unit vector in $\RR^d$ whose $i$-th element is equal to
$1$.
We have
\begin{align*}
 |I^N_{2,1}(t_l)|&\le
C\int_0^{t_l}|Y^N(s)-Y^S(s)|ds\nonumber\\
&\quad +Ct_l
\max_{0\le u\le v\le T, |v-u|\le 2^{-N}T}\left(|Y^N(v)-Y^N(u)|+|Y^S(v)-Y^S(u)|\right),
\end{align*}
\begin{align*}
 |I^N_{2,2}(t_l)|&\le C
\sum_{k=1}^l\max_{0\le u\le v\le T, |v-u|\le 2^{-N}T}
|Y^N(v)-Y^N(u)|\cdot|\Delta_kB^N|^2.
\end{align*}
By Burkholder-Davis-Gundy's inequality, 
we have
\begin{align*}
 E\left[\max_{1\le k\le l}
|I^N_{2,3}(t_k)|^{2}
\right]&\le
CE\left[
\sum_{k=1}^l\eta_k
\right],
\end{align*}
where
\begin{align*}
 \eta_k=\sum_{i=1}^n\left((\xi^i_k)^2-2^{-N}T\right)^2
+\sum_{1\le i<j\le n}(\xi^i_{k})^2(\xi^{j}_k)^2.
\end{align*}
Here $\xi^i_{k}=B^i(t_k)-B^i(t_{k-1})$~$(1\le i\le n)$
which is the increment of
the $i$-th element of the Brownian motion.
By the estimates in Lemma~2.8, Theorem~2.9  and Lemma~4.5 in
 \cite{aida-sasaki} 
and arguing similarly to pages 3813 and 3814 in \cite{aida-sasaki},
we have
\begin{align*}
E\left[\max_{1\le k\le l}|I^N_{2,1}(t_k)|^{2}\right]&\le Ct_l^2
\left(2^{-N}T\right)^{(1-\ep)/6}+Ct_l^{2}\left(2^{-N}T\right)^{1-\ep}\\
E\left[\max_{1\le k\le l}|I^N_{2,2}(t_k)|^{2}\right]&\le
C\left(2^{-N}T\right)^{1-\ep}\\
 E\left[\max_{1\le k\le l}|I_{2,3}^N(t_k)|^{2}\right]
&\le
C\cdot 2^{-N}T.
\end{align*}
Finally, since
$
\max_{0\le t\le T} 
|I_3^N(t)|\le C
\|\Phi^N\|_{[0,T]}\max_{k}|\Delta_kB^N|
$
we have
\begin{align*}
 E\left[\max_{0\le t\le T} 
|I_3^N(t)|^2\right]
&\le (2^{-N}T)^{1-\ep}
\end{align*}
which completes the proof.
\end{proof}

Finally, we discuss the relation 
between the solution to reflected rough differential equation
which is obtained as a limit of the Euler approximation
defined in (\ref{implicit skorohod equation}) and $Y^S$.
For each $B_{s,t}$, we see the existence
of the solution $y^{\Delta}(B,t)$.
However, it is not trivial to see that a certain version of
$y^{\Delta}(B,t)$ is a semimartingale.
Therefore we need the following proposition.

\begin{pro}
Assume $D$ is convex and satisfies {\rm (H1)}.
Let $\{B_t(w)\}$ be an ${\cal F}_t$-Brownian motion
on a probability space $(S,{\cal F},P)$
and $\eta_t(w)$ be a continuous ${\cal F}_t$-semimartingale
 with $E[\|\eta\|_{\infty,[s,t]}^q]\le C_q(t-s)^{q/2}$ for
all $q\ge 1$ and $0\le s\le t\le T$.
We consider the following equation.
\begin{align}
Y_t(w)&=y_0+\eta_t(w)+F\left(\int_0^t\Phi(r,w)\otimes
 dB_r(w)\right)+
\Phi(t,w)\qquad 0\le t\le T,
\label{euler approximation of rrde 2}
\end{align}
where
$Y_t(w)$ is an
${\cal F}_t$-adapted continuous process
and $\Phi(t,w)$ is an
${\cal F}_t$-adapted continuous bounded variation
process,
and
$(Y_t(w), \Phi(t,w))$ is the solution to the Skorohod problem
associated with
$y_0+\eta_t(w)+F\left(\int_0^t\Phi(r,w)\otimes
 dB_r(w)\right)$.
For this problem, there exists a unique solution.
\end{pro}

\begin{proof}
 We consider again an Euler approximation.
Let $\Delta=\{t_k\}$ be a partition of $[0,T]$.
We write $|\Delta|=\max_k(t_k-t_{k-1})$
and $\pi^{\Delta}(t)=\max\{ t_k~|~t_k\le t\}$.
Let $Y^{\Delta}_t$ be the solution to
the Skorohod equation:
\begin{align*}
 Y^{\Delta}_t&=Y^{\Delta}_{t_{k-1}}+\eta_t-\eta_{t_{k-1}}+
F\left(\Phi^{\Delta}(t_{k-1})\otimes(B_t-B_{t_{k-1}})\right)
+\Phi^{\Delta}(t)-\Phi^{\Delta}(t_{k-1})
\qquad t_{k-1}\le t\le t_{k}.
\end{align*}
Then $Y^{\Delta}, \Phi^{\Delta}$ satisfy
\begin{align*}
 Y^{\Delta}_t&=y_0+\eta_t+
\int_0^tF\left(\Phi^{\Delta}(\pi^{\Delta}(t))\otimes dB_t\right)+\Phi^{\Delta}(t).
\end{align*}
Let $q\ge 2$.
By the assumption (H1), we have
\begin{align*}
E\left[\|\Phi^{\Delta}\|_{\infty,[s,t]}^q\right] &\le
C_q(t-s)^{q/2}+C_q(t-s)^{(q-2)/2}\int_s^tE\left[
|\Phi^{\Delta}(\pi^{\Delta}(u))|^q\right]du.
\end{align*}
Hence by considering the case where $s=0$,
we have
\begin{align*}
E\left[\|\Phi^{\Delta}\|_{\infty,[0,t]}^q\right] &\le
C_qt^{q/2}+C_qt^{(q-2)/2}\int_0^t
E\left[\|\Phi^{\Delta}\|_{\infty,[0,u]}^q\right]du
\end{align*}
and by the Gronwall inequality, we get
$E\left[\|\Phi^{\Delta}\|_{\infty,[0,t]}^q\right]\le
 C_qT^{q/2}\exp\left(T^{q/2}\right)$.
Thus, we obtain
\begin{align}
E\left[\|\Phi^{\Delta}\|_{\infty,[s,t]}^q\right]&\le
C_q\left(1+T^{q/2} \exp\left(T^{q/2}\right)\right)(t-s)^{q/2}
\quad 0\le s\le t\le T.\label{estimate for Phidelta}
\end{align}
Let $\Delta'$ be another partition of $[0,T]$.
 Define
\begin{align*}
Z(t)&=Y^{\Delta}(t)-Y^{\Delta'}(t),\\
k(t)&=|Z(t)|^2.
\end{align*}
By the It\^o formula, we have
\begin{align}
dk(t)
&=
2\left(Z(t),F\Bigl(\left(\Phi^{\Delta}(\pi^{\Delta}(t))
-\Phi^{\Delta'}(\pi^{\Delta'}(t))\right)\otimes dB_t\right)\Bigr)\nonumber\\
&+
\sum_{i=1}^d\Bigl |F\left(
\Phi^{\Delta}(\pi^{\Delta}(t))-\Phi^{\Delta'}(\pi^{\Delta'}(t)),e_i
\right)\Bigr |^2dt
\Biggr\}\nonumber\\
&+
2\left(Z(t),d\Phi^{\Delta}(t)-d\Phi^{\Delta'}(t)\right).
\label{dk}
\end{align}
By the convexity of $D$,
we obtain
\begin{align*}
 E\left[|Y^{\Delta}_t-Y^{\Delta'}_t|^2\right]&\le
C_F\int_0^t E\left[|\Phi^{\Delta}(u)-\Phi^{\Delta'}(u)|^2\right]du
+C(|\Delta|+|\Delta'|)t,
\end{align*}
where we have used the estimate (\ref{estimate for Phidelta})
and the positive constant $C_F$ depends on the (Hilbert-Schmidt) 
norm of $F$.
Combining the above inequality and the identity
\begin{align}
 \Phi^{\Delta}(t)-\Phi^{\Delta'}(t)&=Y^{\Delta}(t)-Y^{\Delta'}(t)
-\int_0^tF\Bigl(\left(\Phi^{\Delta}(u)-\Phi^{\Delta'}(u)\right)\otimes dB_u\Bigr),
\label{Phi and Y}
\end{align}
we obtain
\begin{align*}
 E\left[|Y^{\Delta}_t-Y^{\Delta'}_t|^2\right]&\le
C(|\Delta|+|\Delta'|)t+
2C_F\int_0^tE[|Y^{\Delta}_u-Y^{\Delta'}_u|^2]du\nonumber\\
&\quad +
2C_F^2\int_0^t\int_0^u
E[|\Phi^{\Delta}(r)-\Phi^{\Delta'}(r)|^2]drdu.
\end{align*}
Iterating this procedure, we have
\begin{align*}
 E[|Y^{\Delta}_t-Y^{\Delta'}_t|^2]
&\le C(|\Delta|+|\Delta'|)t+
C\int_0^tE[|Y^{\Delta}_s-Y^{\Delta'}_s|^2]ds.
\end{align*}
By the Gronwall inequality, we obtain
\begin{align*}
 E[|Y^{\Delta}_t-Y^{\Delta'}_t|^2]
&\le C(|\Delta|+|\Delta'|)e^{Ct}t.
\end{align*}
Therefore, by (\ref{Phi and Y}),
\begin{align*}
 E\left[|\Phi^{\Delta}(t)-\Phi^{\Delta'}(t)|^2\right]
&\le 2C(|\Delta|+|\Delta'|)e^{Ct}t+C_F\int_0^t
E\left[|\Phi^{\Delta}(s)-\Phi^{\Delta'}(s)|^2\right]ds
\end{align*}
and
\begin{align*}
 E\left[|\Phi^{\Delta}(t)-\Phi^{\Delta'}(t)|^2\right]&\le
2C(|\Delta|+|\Delta'|)e^{(C+C_F)t}t.
\end{align*}
Therefore $L^2$-limit $Y_t:=\lim_{|\Delta|\to 0}Y^{\Delta}_t$
and $\Phi(t):=\lim_{|\Delta|\to 0}\Phi^{\Delta}(t)$ exist.
Moreover there exists a subsequence $\Delta$ such that
$\int_0^tF\left(\Phi^{\Delta}(\pi^{\Delta}(s))\otimes dB(s)\right)$
converges to $\int_0^tF\left(\Phi(s)\otimes dB(s)\right)$ $0\le t\le T$
uniformly $P$-a.s.~$\omega$.
Thus, by the continuity of the Skorohod mapping, we see that
the pair $(Y,\Phi)$ is a solution.
We prove the uniqueness.
Let $(Y,\Phi)$ and $(Y',\Phi')$ be solutions
to (\ref{euler approximation of rrde 2}).
Then by a similar calculation to (\ref{dk}), we have
\begin{align*}
 E\left[|Y(t)-Y'(t)|^2\right]&\le
C_F\int_0^tE[|\Phi(s)-\Phi'(s)|^2]ds.
\end{align*}
By arguing similarly to the above, we complete the proof.
\end{proof}

We consider solutions to
(\ref{implicit skorohod equation}) when
$X_{s,t}=B_{s,t}$.
By applying the above proposition, 
we see
that the solution $(y^{\Delta}(B),\Phi^{\Delta}(B))$
is unique 
in the set of semimartingales.
We obtain the following convergence speed of
$y^{\Delta}(B)$.
Below we denote $y^{\Delta}(B)$ by $y^{\Delta}$ simply.

\begin{thm}
Assume that $D$ is convex and {\rm (H1)} is
satisfied. Let $\Delta_N=\{2^{-N}kT\}_{k=0}^{2^N}$.
There exists a full measure set $\Omega'\subset \Omega$
such that for any $B\in \Omega'$,
$y^{\Delta_N}$ and
$\Phi^{\Delta_N}$ converge to $Y^S$ and $\Phi^{S}$ 
uniformly respectively and
\begin{align}
E\left[\sup_{0\le t\le T}
|y^{\Delta_N}_t-Y^S(t)|^2\right]\le C_T\Delta_N^{4/p}.\label{ydeltaN
 and YS}
\end{align}
\end{thm}

\begin{proof}
Semimartingales $y^{\Delta_N}_t$ and $\Phi^{\Delta_N}(t)$
satisfied the following reflected SDE
\begin{align*}
dy^{\Delta_N}_t&=\sigma\left(y^{\Delta_N}_t\right)\circ dB(t)+
R_N(t)+\Phi^{\Delta_N}(t),\qquad y^{\Delta_N}(B)_0=y_0\\
\end{align*}
where 
\begin{align*}
 R_N(t)&=
\sum_{k=1}^{l-1}M_{t_{k-1},t_k}+M_{t_{l-1},t}
\quad t_{l-1}\le t\le t_l
\end{align*}
and
\begin{align*}
M_{t_{k-1},t}&=
\int_{t_{k-1}}^t\int_{t_{k-1}}^s
\left\{
\Bigl(\left(D\sigma(y^{\Delta_N}_u)\right)(\sigma(y^{\Delta_N}_u))
-\left(D\sigma(y^{\Delta_N}_{t_{k-1}})\right)(\sigma(y^{\Delta_N}_{t_{k-1}}))
\Bigr)
\circ dB(u)\right\}\circ dB(s)\nonumber\\
&\quad +
\int_{t_{k-1}}^{t_k}\int_{t_{k-1}}^s
\left\{\Bigl(\left(D\sigma(y^{\Delta_N}_u)\right)-
\left(D\sigma(y^{\Delta_N}_{t_{k-1}})\right)\Bigr)d\Phi^{\Delta_N}(u)\right\}
\circ dB(s)\qquad t_{k-1}\le t\le t_k.
\end{align*}
$R_N(t)$ is an $\RR^d$-valued semimartingale and its quadratic variation
satisfies that
for any unit vectors $\xi$ and $q\ge 1$,
\begin{align*}
\langle (R_N,\xi)\rangle_T\le C_N \Delta_N^{4/p} ,\\
\sup_NE\left[C_N^q\right]<\infty.
\end{align*}
These estimates follow from Lemma~\ref{estimate on ydelta2}
and the estimate on the control function
$\omega$ of $B_{s,t}$.
Set 
\begin{align*}
 Z^N(t)=Y^S(t)-y^{\Delta_N}_t,\quad
k_N(t)=|Z^N(t)|^2.
\end{align*}
Then by a similar calculation to the proof of Theorem 3.1 in
 \cite{aida-sasaki},
we obtain for $0\le T'\le T$,
\begin{align*}
 E\left[\sup_{0\le t\le T'}k_N(t)\right]&\le
C_T\Delta_N^{4/p}+C_T\int_0^{T'}E\left[\sup_{0\le s\le t}k_N(s)\right]dt
\end{align*}
which shows 
(\ref{ydeltaN and YS}) and
$E\left[\sup_{0\le t\le T}k_N(t)\right]\le C_T'\Delta_N^{4/p}$ and
$\sup_{0\le t\le T}|Y^S(t)-y^{\Delta_N}_t|\to 0$
for almost all $B$.
These estimates imply that
\begin{align*}
\lim_{N\to\infty} \sup_{0\le t\le T}\left|\int_0^t\sigma\left(y^{\Delta_N}(s)\right)\circ
 dB(s)-\int_0^t\sigma(Y^S(s))\circ dB(s)\right|=0
\quad a.s.
\end{align*}
and so we have $\lim_{N\to\infty}
\sup_{0\le t\le T}|\Phi^{\Delta_N}(t)-\Phi^{S}(t)|=0\quad a.s.$
\end{proof}

\noindent
{\bf Acknowledgement}

\noindent
This research was partially supported by Grant-in-Aid for
Scientific Research (B) No.24340023.
The author would like to thank the referee for the 
valuable comments and suggestions which improve the 
quality of the paper.

\end{document}